\documentclass{amsart}

\usepackage{amsthm,amsfonts,amsmath,amssymb}
\usepackage[all]{xy}
\usepackage{lscape}
\usepackage{lunadiagrams}
\usepackage{hyperref}
\usepackage{pdfpages}

\theoremstyle{plain}

\newtheorem{theorem}{Theorem}[section]
\newtheorem{lemma}[theorem]{Lemma}     

\newtheorem{proposition}[theorem]{Proposition}

\theoremstyle{definition}

\theoremstyle{remark}

\newtheorem{remark}[theorem]{Remark}

\numberwithin{equation}{section}

\DeclareMathOperator{\Pic}{Pic}

\DeclareMathOperator{\Lie}{Lie}

\DeclareMathOperator{\height}{ht}
\DeclareMathOperator{\supp}{supp}

 \newcommand{\calL}{\mathcal L}
 
\newcommand{\calO}{\mathcal O} 
\newcommand{\calX}{\mathcal X}

\newcommand{\mC}{\mathbb C} \newcommand{\mN}{\mathbb N}
\newcommand{\mP}{\mathbb P} 
 \newcommand{\mZ}{\mathbb Z}

\newcommand{\gog}{\mathfrak g}

\newcommand{\gok}{\mathfrak k}

\newcommand{\gop}{\mathfrak p}

\newcommand{\got}{\mathfrak t}
\newcommand{\goz}{\mathfrak z}

\newcommand{\sfA}{\mathsf A} 
 \newcommand{\sfD}{\mathsf D}
\newcommand{\sfE}{\mathsf E} \newcommand{\sfF}{\mathsf F}
\newcommand{\sfG}{\mathsf G}

\newcommand{\gra}{\alpha}     \newcommand{\grg}{\gamma}
   \newcommand{\grs}{\sigma}
 \newcommand{\gro}{\omega}

\newcommand{\grG}{\Gamma} \newcommand{\grD}{\Delta}  
 \newcommand{\grS}{\Sigma}

\newcommand{\mss}{\mathrm{ss}}

\newcommand{\ra}         {\rightarrow}
\newcommand{\lra}        {\longrightarrow}

\newcommand{\vuoto}      {\varnothing}

\renewcommand{\geq}      {\geqslant}
\renewcommand{\leq}      {\leqslant}
\newcommand{\senza}      {\smallsetminus}

\newcommand{\ol}         {\overline}

\newcommand{\wt}         {\widetilde}

\newsavebox{\kdwzero}
\newsavebox{\kdwone}
\newsavebox{\kdwtwo}
\savebox{\kdwzero}{\put(-240,240){0}}
\savebox{\kdwone}{\put(-240,240){1}}
\savebox{\kdwtwo}{\put(-240,240){2}}

\newcommand{\esix}[6]{\begin{picture}(2100,1200)\put(100,600){\small #1}\put(900,0){\small #2}\put(500,600){\small #3}\put(900,600){\small #4}\put(1300,600){\small #5}\put(1700,600){\small #6}\end{picture}}
\newcommand{\eseven}[7]{\begin{picture}(2500,1200)\put(100,600){\small #1}\put(900,0){\small #2}\put(500,600){\small #3}\put(900,600){\small #4}\put(1300,600){\small #5}\put(1700,600){\small #6}\put(2100,600){\small #7}\end{picture}}
\newcommand{\eeight}[8]{\begin{picture}(2900,1200)\put(100,600){\small #1}\put(900,0){\small #2}\put(500,600){\small #3}\put(900,600){\small #4}\put(1300,600){\small #5}\put(1700,600){\small #6}\put(2100,600){\small #7}\put(2500,600){\small #8}\end{picture}}
\renewcommand{\ffour}[4]{\begin{picture}(1700,600)\put(100,0){\small #1}\put(500,0){\small #2}\put(900,0){\small #3}\put(1300,0){\small #4}\end{picture}}

\title[Spherical nilpotent orbits in complex symmetric pairs]{Regular functions on spherical nilpotent\\orbits in complex symmetric pairs:\\exceptional cases}

\email{bravi@mat.uniroma1.it}

\curraddr{\textsc{Dipartimento di Matematica\\ Sapienza Universit\`a di Roma\\ 
Piazzale Aldo Moro 5\\ 00185 Roma, Italy}}

\email{jacopo.gandini@sns.it}

\curraddr{\textsc{Scuola Normale Superiore\\
Piazza dei Cavalieri 7\\ 56126 Pisa, Italy}}

\author{Paolo Bravi, Jacopo Gandini}

\begin{document}

\begin{abstract}
Given an exceptional simple complex algebraic group $G$ and a symmetric pair $(G,K)$, we study the spherical nilpotent $K$-orbit closures in the isotropy representation of $K$. We show that they are all normal except in one case in type $\sfG_2$, and compute the $K$-module structure of the ring of regular functions on their normalizations.
\end{abstract}

\maketitle

\section*{Introduction}

In the present paper we complete the study of the spherical nilpotent orbits in complex symmetric spaces by considering the symmetric pairs $(\mathfrak g,\mathfrak k)$ with $\mathfrak g$ of exceptional type, the cases with $\mathfrak g$ of classical type being considered in \cite{BCG} and \cite{BG}. We refer to those papers for some background and motivation of this work. 

We keep here the notation introduced therein. In particular, here $G$ is a connected simple complex algebraic group of exceptional type, $K$ the fixed point subgroup of an involution $\theta$ of $G$, and $\mathfrak g=\mathfrak k\oplus\mathfrak p$ the corresponding eigenspace decomposition.
 
The spherical nilpotent $K$-orbits in $\mathfrak p$ have been classified by King in \cite{Ki04}. In the exceptional cases such classification is based on \makebox[0pt]{\rule{3pt}{0pt}\rule[4pt]{3pt}{0.8pt}}Dokovi\'c's tables of nilpotent orbits in simple exceptional real Lie algebras \cite{D88a,D88b}.

We give the list of the spherical nilpotent orbits $Ke \subset \mathfrak p$, together with a normal triple $\{e,h,f\}$ and a description of the centralizer of $e$ in $\mathfrak k$.

The normalizer of the centralizer of $e$, $\mathrm N_K(K_e)$, is a wonderful subgroup of $K$. We compute its Luna spherical system and study the surjectivity of the multiplication of sections of globally generated line bundles on the corresponding wonderful variety.

In particular, we obtain that for the wonderful varieties arising from the spherical nilpotent orbits in exceptional symmetric pairs such multiplication is always surjective (Theorem~\ref{teo: projnorm}).

We use this to study the normality of the closure of the spherical nilpotent orbit $Ke$ in $\mathfrak p$ and compute the weight semigroup of its normalization, and we obtain that in the exceptional cases all the spherical nilpotent orbit closures are normal except in one case in type $\mathsf G_2$, see Theorem~\ref{teo:normal}.

In Section~\ref{s:1} we compute the Luna spherical systems. In Section~\ref{s:2} we prove the surjectivity of the multiplication. In Section~\ref{s:3} we deduce our results on normality of orbit closures and compute the weight semigroups.

In Appendix~\ref{A} we report the list of the orbits under consideration, together with normal triples and centralizers. In Appendix~\ref{B} we put the tables with the results of our computations.

%\subsection*{Notation}
%
%As in our previous papers, simple roots of irreducible root systems are denoted by $\alpha_1,\alpha_2,\ldots$ and %enumerated as in Bourbaki, when belonging to different irreducible components they are denoted by $\alpha_1,\alpha_2,%\ldots$, $\alpha'_1,\alpha'_2,\ldots$, $\alpha''_1,\alpha''_2,\ldots$, and so on. When $G$ (resp. $K$, $T$, ...) is an %algebraic group, we will denote the associated Lie algebra by the corresponding fraktur character $\gog$ (resp. $\gok$, %$\got$, ...). If moreover $X$ is a $G$-variety (resp. a $K$-variety) and $x \in X$, we will denote the stabilizer of $x$ by %$G_x$ (resp. $K_x$).

%%%%%%%%%%%%%%%%%%%%%%%%%%%%%%%%%%%%%%%%%%%%%%%%%%
%%%%%%%%%%%%%%%%%%%%%%%%%%%%%%%%%%%%%%%%%%%%%%%%%%
\section{Spherical systems}\label{s:1}
%%%%%%%%%%%%%%%%%%%%%%%%%%%%%%%%%%%%%%%%%%%%%%%%%%
%%%%%%%%%%%%%%%%%%%%%%%%%%%%%%%%%%%%%%%%%%%%%%%%%%

In this first section, for every spherical nipotent orbit $Ke$ in $\mathfrak p$, we compute the spherical system of $\mathrm N_K(K_e)$, the normalizer of the centralizer, which is a wonderful subgroup of $K$ (see \cite[Section 1]{BCG} for background on wonderful subgroups and wonderful varieties). We take the spherical system given in the tables of Appendix~\ref{B} and explain case-by-case (in most cases we just provide references) that it actually corresponds to the normalizer of $Ke$ which is described in the list in Appendix~\ref{A}.

We keep the notation introduced in \cite{BCG}.

The rank zero cases correspond to parabolic subgroups: 1.1; 2.1; 3.1, 3.2, 3.5; 4.1; 5.1; 6.1; 7.1, 7.2, 7.5; 8.1; 9.1; 10.1; 11.1; 12.1. Notice that if we take the parabolic subgroups containing the opposite of the fixed Borel subgroup, we get that the root subsystem of their Levi factor (containing the fixed maximal torus) is generated by $S^\mathrm p$. This means that, if we take the standard parabolic subgroups (i.e.\ containing the fixed Borel subgroup), we get that the root subsystem of their Levi factor is generated by $(S^\mathrm p)^*$.
In the list all the given parabolic subgroups are standard, by construction.

For the positive rank cases we take the localization of the spherical system on $\supp_S\Sigma$, the support of the spherical roots of the spherical system, which corresponds to a wonderful subgroup of a Levi subgroup $M$ of $K$ that we describe. The investigated wonderful subgroup of $K$ can be obtained by parabolic induction from the wonderful subgroup of $M$. 

The localization on $\supp_S\Sigma$ is a well-known symmetric case (see \cite{BP15}) in: 1.2; 2.2, 2.3; 3.3, 3.4, 3.6; 5.2--5.4; 6.2, 6.3; 7.3, 7.4, 7.6--7.10; 8.2; 9.2, 9.3; 10.2, 10.3; 12.2. It corresponds to a wonderful reductive (but not symmetric) subgroup of $M$ (see \cite{BP15}) in the cases 3.9 and 11.2. It corresponds to a comodel case (see \cite[Section~5]{BGM}) in the cases 5.8, 5.9 and 8.6.

The remaining cases are somewhat all very similar: indeed, they all possess a positive color, giving a morphism onto another spherical system of smaller rank which is very easy to describe (see \cite[\S2.3]{BP16}). Moreover, in all these cases the target of the morphism is always a parabolic induction of a symmetric case. A positive color is by definition an element $D\in\Delta$ which takes nonnegative values on all the spherical roots (through the Cartan pairing $c\colon\Delta\times\Sigma\to\mathbb Z$). 
Every positive color provides a distinguished subset of colors by itself, and the corresponding quotient has $\Delta\setminus\{D\}$ as set of colors and $\{\sigma\in\Sigma\,:\,c(D,\sigma)=0\}$ as set of spherical roots.

To describe the subgroups of $M$, we fix a maximal torus and a Borel subgroup in $K$, and a corresponding set of simple roots $\alpha_1,\ldots,\alpha_n;\alpha'_1,\ldots,\alpha'_{n'}; \ldots$ for all almost-simple factors of $K$. The parabolic subgroups of $M$ are all chosen to be standard. They are in correspondence with subsets of simple roots, the simple roots that generate the root subsystem of the corresponding Levi factor. To be as explicit as possible, we work in the semisimple part $M'$ of $M$, and when $M$ is of classical type, we take $M'$ to be a classical matrix group. 

Let us fix some further notation. Since it is always a nontrivial parabolic induction, the wonderful subgroup of $M'$ corresponding to the target is given as $\tilde L P^\mathrm u$, where $P=L\,P^\mathrm u$ is a parabolic subgroup of $M'$ and $\tilde L\subset L$. The wonderful subgroup corresponding to the source $H=L_HH^\mathrm u$ can be included in $\tilde L P^\mathrm u$, with $L_H\subset \tilde L$ and $H^\mathrm u\subset P^\mathrm u$. 

\subsubsection*{Cases 1.3 and 10.4}
\[\rule[-6pt]{0pt}{6pt}\begin{picture}(4200,1200)(-300,-300)\put(0,0){\usebox{\dynkincthree}}\multiput(0,0)(1800,0){2}{\usebox{\aprime}}\put(3600,0){\usebox{\aone}}\end{picture}\quad\longrightarrow\quad\begin{picture}(4200,1200)(-300,-300)\put(0,0){\usebox{\dynkincthree}}\multiput(0,0)(1800,0){2}{\usebox{\aprime}}\put(3600,0){\usebox{\wcircle}}\end{picture}\]
 Take $M'=\mathrm{Sp}(6)$, $P$ the parabolic subgroup of $M'$ corresponding to $\alpha_1$, $\alpha_2$, and take $\tilde L$ to be the normalizer of $\mathrm{SO}(3)$ in $L\cong\mathrm{GL}(3)$.

The wonderful subgroup $H$ corresponding to the source is given by the same Levi factor $L_H=\tilde L$, and unipotent radical $H^\mathrm u$ of codimension 1 in $P^\mathrm u$. The unipotent radical $P^\mathrm u$ of $P$ is a simple $L$-module of highest weight $2\omega_{\alpha_1}$ which splits into two simple $\tilde L$-submodules of dimension 5 and 1, respectively, so that $H^\mathrm u$ is uniquely determined.  

\subsubsection*{Cases 2.4 and 5.5}
\[\rule[-12pt]{0pt}{12pt}\begin{picture}(7800,1200)(-300,-300)
\put(0,0){\usebox{\dynkinafive}}
\multiput(0,0)(5400,0){2}{\multiput(0,0)(1800,0){2}{\usebox{\wcircle}}}
\put(3600,0){\usebox{\aone}}
\multiput(0,-1500)(7200,0){2}{\line(0,1){1200}}
\put(0,-1500){\line(1,0){7200}}
\multiput(1800,-1200)(3600,0){2}{\line(0,1){900}}
\put(1800,-1200){\line(1,0){3600}}
\end{picture}
\quad\longrightarrow\quad
\begin{picture}(7800,1200)(-300,-300)
\put(0,0){\usebox{\dynkinafive}}
\multiput(0,0)(5400,0){2}{\multiput(0,0)(1800,0){2}{\usebox{\wcircle}}}
\put(3600,0){\usebox{\wcircle}}
\multiput(0,-1500)(7200,0){2}{\line(0,1){1200}}
\put(0,-1500){\line(1,0){7200}}
\multiput(1800,-1200)(3600,0){2}{\line(0,1){900}}
\put(1800,-1200){\line(1,0){3600}}
\end{picture}\]
 Take $M'=\mathrm{SL}(6)$, $P$ the parabolic subgroup of $M'$ corresponding to $\alpha_1$, $\alpha_2$, $\alpha_4$, $\alpha_5$, and take $\tilde L$ to be the normalizer of $\mathrm{SL}(3)$ embedded diagonally into $L\cong\mathrm{S}(\mathrm{GL}(3)\times\mathrm{GL}(3))$. 

The wonderful subgroup $H$ is given by the same Levi factor $L_H = \tilde L$, and unipotent radical $H^\mathrm u$ of codimension 1 in $P^\mathrm u$. The unipotent radical $P^\mathrm u$ of $P$ is a simple $L$-module of highest weight $\omega_{\alpha_1}+\omega_{\alpha_5}$ which splits into two simple $\tilde L$-submodules of dimension 8 and 1, respectively, so that $H^\mathrm u$ is uniquely determined.  

\subsubsection*{Case 2.5}
\[\rule[-9pt]{0pt}{9pt}\begin{picture}(6900,1500)(-300,-300)\put(0,0){\usebox{\dynkinathree}}\multiput(0,0)(3600,0){2}{\usebox{\wcircle}}\multiput(0,-1200)(3600,0){2}{\line(0,1){900}}\put(0,-1200){\line(1,0){3600}}\multiput(1800,0)(4500,0){2}{\usebox{\aone}}\multiput(1800,1200)(4500,0){2}{\line(0,-1){300}}\put(1800,1200){\line(1,0){4500}}\end{picture}\quad\longrightarrow\quad\begin{picture}(6900,1500)(-300,-300)\put(0,0){\usebox{\dynkinathree}}\multiput(0,0)(3600,0){2}{\usebox{\wcircle}}\multiput(0,-1200)(3600,0){2}{\line(0,1){900}}\put(0,-1200){\line(1,0){3600}}\multiput(1800,0)(4500,0){2}{\usebox{\wcircle}}\put(6300,0){\usebox{\vertex}}\end{picture}\]
 Take $M'=\mathrm{SL}(4)\times\mathrm{SL}(2)$, $P$ the parabolic subgroup of $M'$ corresponding to $\alpha_1$, $\alpha_3$, and take $\tilde L$ to be the normalizer of $\mathrm{SL}(2)$ embedded diagonally into $L\cong\mathrm{S}(\mathrm{GL}(2)\times\mathrm{GL}(2))$.

The semisimple parts of the Levi subgroups $L_H$ and $\tilde L$ are equal, while the center of $L_H$ has codimension 1 in the center of $\tilde L$. The unipotent radical $H^\mathrm u$ has codimension 1 in $P^\mathrm u$. The unipotent radical $P^\mathrm u$ of $P$ is the direct sum of two simple $L$-modules of highest weight $\omega_{\alpha_1}+\omega_{\alpha_3}$ and $0$, respectively. The former splits into two simple $\tilde L$-submodules of dimension 3 and 1, respectively, so that $H^\mathrm u$ is the $L_H$-complement of a 1 dimensional submodule of the two 1 dimensional $\tilde L$-submodules of $P^\mathrm u$, that projects nontrivially on both summands. The subgroup $H$ is uniquely determined up to conjugation.  

\subsubsection*{Cases 3.7 and 3.8}
\[\rule[-12pt]{0pt}{12pt}\begin{picture}(3600,1800)(-300,-300)
\put(0,0){\usebox{\dynkindfour}}
\put(0,0){\usebox{\athreene}}
\put(0,0){\usebox{\athreese}}
\end{picture}
\quad\longrightarrow\quad
\begin{picture}(3600,1800)(-300,-300)
\put(0,0){\usebox{\dynkindfour}}
\multiput(3000,1200)(0,-2400){2}{\usebox{\wcircle}}
\end{picture}\]
 Take $M'=\mathrm{SO}(8)$, $P$ the parabolic subgroup of $M'$ corresponding to $\alpha_1$, $\alpha_2$, and take $\tilde L=L\cong\mathrm{GL}(3)\times\mathrm{GL}(1)$.

The semisimple parts of the Levi subgroups $L_H$ and $L$ are equal, while the center of $L_H$ has codimension 1 in the center of $L$. The unipotent radical $H^\mathrm u$ has codimension 3 in $P^\mathrm u$. The unipotent radical $P^\mathrm u$ of $P$ is the direct sum of three simple $L$-modules, all of them of dimension 3. With respect to the semisimple part of $L$, two of them have highest weight $\omega_{\alpha_1}$ and the other one has highest weight $\omega_{\alpha_2}$. Therefore $H^\mathrm u$ is the $L_H$-complement of a 3 dimensional submodule of the two $L$-submodules of $P^\mathrm u$ of highest weight $\omega_{\alpha_1}$ that projects nontrivially on both summands. The subgroup $H$ is uniquely determined up to conjugation.  

\subsubsection*{Cases 6.4 and 8.3}
\[\rule[-21pt]{0pt}{21pt}\begin{picture}(6900,1500)(0,-300)
\put(0,0){\usebox{\dynkindsix}}
\multiput(1800,0)(3600,0){2}{\usebox{\gcircle}}
\put(6600,-1200){\usebox{\aone}}
\end{picture}
\quad\longrightarrow\quad
\begin{picture}(6900,1500)(0,-300)
\put(0,0){\usebox{\dynkindsix}}
\multiput(1800,0)(3600,0){2}{\usebox{\gcircle}}
\put(6600,-1200){\usebox{\wcircle}}
\end{picture}\]
 Take $M'=\mathrm{SO}(12)$, $P$ the parabolic subgroup of $M'$ corresponding to $\alpha_1,\ldots,\alpha_5$, and take $\tilde L$ to be the normalizer of $\mathrm{Sp}(6)$ in $L\cong\mathrm{GL}(6)$. 

The wonderful subgroup $H$ is given by the same Levi factor $L_H = \tilde L$, and unipotent radical $H^\mathrm u$ of codimension 1 in $P^\mathrm u$. The unipotent radical $P^\mathrm u$ of $P$ is a simple $L$-module of highest weight $\omega_{\alpha_2}$ which splits into two simple $\tilde L$-submodules of dimension 14 and 1, respectively, so that $H^\mathrm u$ is uniquely determined.  

\subsubsection*{Cases 6.5 and 9.5}
\[\rule[-9pt]{0pt}{9pt}
\begin{picture}(11400,1500)(-300,-300)\multiput(0,0)(2700,0){2}{\usebox{\aone}}\multiput(0,1200)(2700,0){2}{\line(0,-1){300}}\put(0,1200){\line(1,0){2700}}\put(2700,0){\usebox{\plusdm}}\end{picture}
\quad\longrightarrow\quad
\begin{picture}(11400,1500)(-300,-300)\multiput(0,0)(2700,0){2}{\usebox{\vertex}}\multiput(0,0)(2700,0){2}{\usebox{\wcircle}}\put(2700,0){\usebox{\plusdm}}\end{picture}\]
 Take $M'=\mathrm{SL}(2)\times\mathrm{SO}(2n)$, $P$ the parabolic subgroup of $M'$ corresponding to $\alpha'_2,\ldots,\alpha'_n$, and take $\tilde L$ to be the normalizer of $\mathrm{SO}(2n-3)$ in $L\cong\mathrm{GL}(1)\times\mathrm{GL}(1)\times\mathrm{SO}(2n-2)$.

The semisimple parts of the Levi subgroups $L_H$ and $\tilde L$ are equal, while the center of $L_H$ has codimension 1 in the center of $\tilde L$. The unipotent radical $H^\mathrm u$ has codimension 1 in $P^\mathrm u$. The unipotent radical $P^\mathrm u$ of $P$ is the direct sum of two simple $L$-modules of highest weight $0$ and $\omega_{\alpha'_2}$, respectively. The latter splits into two simple $\tilde L$-submodules of dimension $2n-3$ and $1$, respectively, so that $H^\mathrm u$ is the $L_H$-complement of a 1-dimensional submodule of the two 1-dimensional $\tilde L$-submodules of $P^\mathrm u$, that projects nontrivially on both summands.   

\subsubsection*{Cases 7.11 and 7.12}
\[\begin{array}{ccc}\rule[-18pt]{0pt}{18pt}
\begin{picture}(7800,1200)(-300,-300)
\put(0,0){\usebox{\dynkinesix}}
\put(1800,0){\usebox{\afour}}
\multiput(7200,0)(-3600,-1800){2}{\circle{600}}
\multiput(7200,0)(-25,-25){13}{\circle*{70}}
\put(6900,-300){\multiput(0,0)(-300,0){10}{\multiput(0,0)(-25,25){7}{\circle*{70}}}\multiput(-150,150)(-300,0){10}{\multiput(0,0)(-25,-25){7}{\circle*{70}}}}
\multiput(3600,-1800)(25,25){13}{\circle*{70}}
\put(3900,-1500){\multiput(0,0)(0,300){4}{\multiput(0,0)(-25,25){7}{\circle*{70}}}\multiput(-150,150)(0,300){4}{\multiput(0,0)(25,25){7}{\circle*{70}}}}
\put(0,0){\usebox{\aone}}
\end{picture}
&
\longrightarrow
&
\begin{picture}(7800,1200)(-300,-300)
\put(0,0){\usebox{\dynkinesix}}
\put(1800,0){\usebox{\wcircle}}
\put(3600,-1800){\usebox{\wcircle}}
\put(0,0){\usebox{\aone}}
\end{picture}\\
\downarrow & \rule[-12pt]{0pt}{30pt} & \downarrow \\
\rule[-18pt]{0pt}{18pt}
\begin{picture}(7800,600)(-300,-300)
\put(0,0){\usebox{\dynkinesix}}
\put(1800,0){\usebox{\afour}}
\multiput(7200,0)(-3600,-1800){2}{\circle{600}}
\multiput(7200,0)(-25,-25){13}{\circle*{70}}
\put(6900,-300){\multiput(0,0)(-300,0){10}{\multiput(0,0)(-25,25){7}{\circle*{70}}}\multiput(-150,150)(-300,0){10}{\multiput(0,0)(-25,-25){7}{\circle*{70}}}}
\multiput(3600,-1800)(25,25){13}{\circle*{70}}
\put(3900,-1500){\multiput(0,0)(0,300){4}{\multiput(0,0)(-25,25){7}{\circle*{70}}}\multiput(-150,150)(0,300){4}{\multiput(0,0)(25,25){7}{\circle*{70}}}}
\put(0,0){\usebox{\wcircle}}
\end{picture}
&
\longrightarrow
&
\begin{picture}(7800,600)(-300,-300)
\put(0,0){\usebox{\dynkinesix}}
\put(1800,0){\usebox{\wcircle}}
\put(3600,-1800){\usebox{\wcircle}}
\put(0,0){\usebox{\wcircle}}
\end{picture}
\end{array}\]
Let us look at the morphism given by the first line of the diagram. Here $M'$ is semisimple of type $\mathsf E_6$, $P$ is the parabolic subgroup of $M'$ corresponding to $\alpha_1$, $\alpha_3$, $\alpha_4$, $\alpha_6$, and take $\tilde L\subset L$ to be the whole factor of type $\mathsf A_3$, times a torus in the factor of type $\mathsf A_1$, times the center of $L$.

The semisimple parts of the Levi subgroups $L_H$ and $\tilde L$ are equal, while the center of $L_H$ has codimension 1 in the center of $\tilde L$. The unipotent radical $H^\mathrm u$ has codimension 4 in $P^\mathrm u$. Looking at the entire commutative diagram of morphisms given by positive colors,
one can see that $H^\mathrm u$ must be the $L_H$-complement of a 4 dimensional submodule of the two $\tilde L$-submodules of $P^\mathrm u$ of lowest weight $\alpha_2$ and $\alpha_5$, respectively, that projects nontrivially on both summands. 

\subsubsection*{Case 9.4}
\[\rule[-15pt]{0pt}{15pt}
\begin{picture}(9600,1200)(-300,-300)
\put(0,0){\usebox{\dynkineseven}}
\multiput(0,0)(7200,0){2}{\usebox{\gcircle}}
\put(9000,0){\usebox{\aone}}
\end{picture}
\quad\longrightarrow\quad
\begin{picture}(9600,1200)(-300,-300)
\put(0,0){\usebox{\dynkineseven}}
\multiput(0,0)(7200,0){2}{\usebox{\gcircle}}
\put(9000,0){\usebox{\wcircle}}
\end{picture}
\]
Here $M'$ is semisimple of type $\mathsf E_7$, $P$ is the parabolic subgroup of $M'$ corresponding to $\alpha_1,\ldots,\alpha_6$, and take $\tilde L$ to be the normalizer of $\mathsf F_4$ in $L$. 

The wonderful subgroup $H$ is given by the same Levi factor $L_H = \tilde L$, and unipotent radical $H^\mathrm u$ of codimension 1 in $P^\mathrm u$. The unipotent radical $P^\mathrm u$ of $P$ is a simple $L$-module of highest weight $\omega_{\alpha_1}$ which splits into two simple $\tilde L$-submodules of dimension 26 and 1, respectively, so that $H^\mathrm u$ is uniquely determined.  

\subsubsection*{Case 10.5}
\[\rule[-6pt]{0pt}{6pt}
\begin{picture}(5100,1500)(-300,-300)\multiput(0,0)(2700,0){2}{\usebox{\aone}}\multiput(0,1200)(2700,0){2}{\line(0,-1){300}}\put(0,1200){\line(1,0){2700}}\put(2700,0){\usebox{\dynkinbtwo}}\put(4500,0){\usebox{\aprime}}\end{picture}
\quad\longrightarrow\quad
\begin{picture}(5100,1500)(-300,-300)\multiput(0,0)(2700,0){2}{\usebox{\wcircle}}\multiput(0,0)(2700,0){2}{\usebox{\vertex}}\put(2700,0){\usebox{\dynkinbtwo}}\put(4500,0){\usebox{\aprime}}\end{picture}\]
 Take $M'=\mathrm{SL}(2)\times\mathrm{SO}(5)$, $P$ the parabolic subgroup of $M'$ corresponding to $\alpha'_2$, and take $\tilde L$ to be the normalizer of $\mathrm{SO}(2)$ in $L\cong\mathrm{GL}(1)\times\mathrm{GL}(1)\times\mathrm{SO}(3)$.

The semisimple parts of the Levi subgroups $L_H$ and $\tilde L$ are equal, while the center of $L_H$ has codimension 1 in the center of $\tilde L$. The unipotent radical $H^\mathrm u$ has codimension 1 in $P^\mathrm u$. The unipotent radical $P^\mathrm u$ of $P$ is the direct sum of two simple $L$-modules of highest weight $0$ and $\omega_{\alpha'_2}$, respectively. The latter splits into two simple $\tilde L$-submodules of dimension $2$ and $1$, respectively, so that $H^\mathrm u$ is the $L_H$-complement of a 1 dimensional submodule of the two 1 dimensional $\tilde L$-submodules of $P^\mathrm u$, that projects nontrivially on both summands.

%%%%%%%%%%%%%%%%%%%%%%%%%%%%%%%%%%%%%%%%%%%%%%%%%%
%%%%%%%%%%%%%%%%%%%%%%%%%%%%%%%%%%%%%%%%%%%%%%%%%%
\section{Projective normality}\label{s:2}
%%%%%%%%%%%%%%%%%%%%%%%%%%%%%%%%%%%%%%%%%%%%%%%%%%
%%%%%%%%%%%%%%%%%%%%%%%%%%%%%%%%%%%%%%%%%%%%%%%%%%

Let $\gop = \bigoplus_{i=1}^N \gop_i$ be the decomposition of $\gop$ into irreducible $K$-modules: recall that $N$ can only be equal to $1$ or $2$. If $(G,K)$ is of Hermitian type, then $N=2$ and $\gop_1 \simeq \gop_2^*$ are dual non-isomorphic $K$-modules. Otherwise, $N=1$ and $\mathfrak p$ is irreducible. As in \cite{BCG} and \cite{BG}, to any spherical nilpotent orbit $Ke \subset \gop$ we associate a wonderful $K$-variety $X$ as follows.

Let $e\in\mathfrak p$ be a nonzero nilpotent element, and write $e = \sum_{i=1}^N e_i$ with $e_i \in \gop_i$. Up to reordering the $K$-modules $\gop_i$, we can assume that $e_i \neq 0$ if and only if $i \leq M$, for some $M \leq N$. If $v = \sum_{i=1}^M v_i$ with $v_i \in \gop_i \senza \{0\}$, let $[v_i] \in \mP(\gop_i)$ be the line defined by $v_i$ and set $\pi(v) =  ([v_1], \ldots, [v_M])$: then we get a morphism
$$\pi : Ke \ra \mP(\gop_1) \times \ldots \times \mP(\gop_M).$$
Moreover, $\ol{Ke}$ is the multicone over $\ol{K\pi(e)}$ (see \cite[Proposition 4.2]{BG}), and if $Ke$ is spherical then the stabilizer of $\pi(e)$ coincides with the normalizer $\mathrm{N}_K(K_e)$ (see \cite[Proposition 1.1]{BG}). Therefore the spherical orbit $K\pi(e)$ admits a wonderful compactification $X$ (see \cite[Section 1]{BG} and the references therein).

%Let $e \in \gop$ be a nilpotent element and suppose that $Ke$ is spherical, we study in this section the normality of the closure $\ol{Ke}$, as well as the $K$-module structure of the coordinate ring $\mC[\wt{Ke}]$ of its normalization. Thanks to the results of previous sections this reduces to a combinatorial problem on the wonderful $K$-variety $X$ that is naturally associated to $\ol{Ke}$.

If $\calL, \calL' \in \Pic(X)$ are globally generated line bundles, we denote by
$$
	m_{\calL,\calL'} : \grG(X,\calL) \otimes \grG(X,\calL') \lra \grG(X,\calL \otimes \calL')  
$$
the multiplication of sections. In this section we prove the following.

\begin{theorem}	\label{teo: projnorm}
Let $(\mathfrak g,\mathfrak k)$ be an exceptional symmetric pair, let $\calO \subset \gop$ be a spherical nilpotent $K$-orbit and let $X$ be the wonderful $K$-variety associated to $\calO$. Then $m_{\calL,\calL'}$ is surjective for all globally generated line bundles $\calL, \calL' \in \Pic(X)$.
\end{theorem}

\subsection{General reductions}\label{ss:General reductions}

As already explained in our previous papers (see \cite[Section 2]{BCG} and \cite[Section 3]{BG}), to prove that the multiplication of sections of globally generated line bundles on a wonderful variety $X$ is surjective is enough to show that $X$ can be obtained by operations of localization, quotient and parabolic induction from another wonderful variety $Y$ for which the surjectivity of the multiplication holds.

Here we show that in order to prove Theorem~\ref{teo: projnorm} it is enough to check the surjectivity of the multiplication in the following four cases. Indeed, we check that all the wonderful varieties associated with spherical nilpotent orbits in exceptional symmeric pairs can be obtained via operations of localization, quotient and parabolic induction from one of these four basic cases, or from other cases for which the surjectivity of the multiplication is already known.

\begin{equation}\tag{\textbf{A}}
\rule[-6pt]{0pt}{6pt}
\begin{picture}(4200,1200)(-300,-300)
\put(0,0){\usebox{\dynkincthree}}
\put(0,0){\usebox{\aprime}}
\put(1800,0){\usebox{\aprime}}
\put(3600,0){\usebox{\aone}}
\end{picture}
\end{equation} 

\begin{equation}\tag{\textbf{B}}
\rule[-12pt]{0pt}{12pt}
\begin{picture}(3600,1800)(-300,-300)
\put(0,0){\usebox{\dynkindfour}}
\put(0,0){\usebox{\athreene}}
\put(0,0){\usebox{\athreese}}
\put(1800,0){\usebox{\athreebifurc}}
\end{picture}
\end{equation} 

\begin{equation}\tag{\textbf{C}}
\rule[-18pt]{0pt}{18pt}
\begin{picture}(7800,1200)(-300,-300)
\put(0,0){\usebox{\dynkinesix}}
\put(0,0){\usebox{\afour}}
\multiput(0,0)(3600,-1800){2}{\circle{600}}
\multiput(0,0)(25,-25){13}{\circle*{70}}
\put(300,-300){\multiput(0,0)(300,0){10}{\multiput(0,0)(25,25){7}{\circle*{70}}}\multiput(150,150)(300,0){10}{\multiput(0,0)(25,-25){7}{\circle*{70}}}}
\multiput(3600,-1800)(-25,25){13}{\circle*{70}}
\put(3300,-1500){\multiput(0,0)(0,300){4}{\multiput(0,0)(25,25){7}{\circle*{70}}}\multiput(150,150)(0,300){4}{\multiput(0,0)(-25,25){7}{\circle*{70}}}}
\put(7200,0){\usebox{\aone}}
\end{picture}
\end{equation} 

\begin{equation}\tag{\textbf{D}}
\rule[-18pt]{0pt}{18pt}
\begin{picture}(9600,1200)(-300,-300)
\put(0,0){\usebox{\dynkineseven}}
\multiput(0,0)(7200,0){2}{\usebox{\gcircle}}
\put(9000,0){\usebox{\aone}}
\end{picture}
\end{equation} 

Let $\calO \subset \gop$ be a spherical nilpotent $K$-orbit in an exceptional symmetric pair and let $X$ be the corresponding wonderful variety, with set of spherical roots $\grS$. When $X$ is a flag variety, or equivalently $\Sigma=\vuoto$, the surjectivity of the multiplication is trivial.

The surjectivity of the multiplication on $X$ is reduced to the surjectivity of the multiplication on the localization of $X$, that we denote by $Z$, at the subset $\supp_S \grS\subset S$. These localizations are described in Section~\ref{s:1}.

\subsubsection{Symmetric cases} In the cases 1.2, 2.2, 2.3, 3.3, 3.4, 3.6, 5.2--5.4, 6.2, 6.3, 7.3, 7.4, 7.6--7.10, 8.2, 9.2, 9.3, 10.2, 10.3 the wonderful variety $Z$ is the wonderful compactification of an adjoint symmetric variety, and the surjectivity of the multiplication holds thanks to \cite{CM_projective-normality}.

\subsubsection{Rank one cases} In the cases 11.2 and 12.2 the wonderful variety
$Z$ is a rank one wonderful variety which is homogeneous under its automorphism group (see \cite{Akh}). Therefore in these cases $Z$ is a flag variety for its automorphism group, and the surjectivity of the multiplication is trivial.

\subsubsection{Comodel cases} In the cases 2.4, 2.5, 5.5, 5.8, 5.9, 6.4, 6.5, 8.3, 8.6, 9.5, 10.5, the wonderful variety $Z$ is a localization of a quotient of a wonderful subvariety of a wonderful variety $Y$ for which the surjectivity of the multiplication holds. 

In particular, in all these cases but the last one, we show that we can take $Y$ to be a comodel wonderful variety, for which the surjectivity holds thanks to \cite[Theorem~5.2]{BGM}.

In the cases 5.8, 5.9 and 8.6 the wonderful variety $Z$ itself is a comodel wonderful variety, of cotype $\mathsf D_7$ in the cases 5.8 and 5.9, and of cotype $\mathsf E_8$ in the case 8.6. 

Let $Y$ be the comodel wonderful variety of cotype $\sfE_8$,
which is the wonderful variety with
the following spherical system for a group of semisimple type $\sfD_7$.
\[
\begin{picture}(9600,4800)(-300,-2400)
\put(0,0){\usebox{\dynkindseven}}
\put(-1800,0){
\multiput(1800,0)(1800,0){5}{\usebox{\aone}}
\multiput(10200,-1200)(0,2400){2}{\usebox{\aone}}
\put(1500,-2700){
\multiput(2100,5100)(3600,0){2}{\line(0,-1){1500}}
\put(2100,5100){\line(1,0){7200}}
\put(9300,5100){\line(0,-1){3000}}
\put(9300,2100){\line(-1,0){300}}
\multiput(300,4050)(7200,0){2}{\line(0,-1){450}}
\put(300,4050){\line(1,0){1700}}
\put(2200,4050){\line(1,0){3400}}
\put(5800,4050){\line(1,0){1700}}
\put(3900,800){\line(0,1){1000}}
%\put(3900,300){\line(1,0){5400}}
%\put(9300,300){\line(0,1){3000}}
%\put(9300,3300){\line(-1,0){300}}
\put(300,300){\line(0,1){1500}}
\put(300,300){\line(1,0){9300}}
\put(9600,3300){\line(-1,0){200}}
\put(9200,3300){\line(-1,0){200}}
\put(9600,300){\line(0,1){3000}}
\put(3900,800){\line(1,0){4500}}
\multiput(2100,1500)(5400,0){2}{\line(0,1){300}}
\put(2100,1500){\line(1,0){1700}}
\put(4000,1500){\line(1,0){3500}}
}
\multiput(1800,600)(1800,0){2}{\usebox{\toe}}
\put(1800,0){
\put(5400,600){\usebox{\tow}}
\put(5400,600){\usebox{\toe}}
\put(7200,600){\usebox{\tone}}
\put(8400,-600){\usebox{\tonw}}
}
}
\end{picture}
\]

If we consider the wonderful subvariety of $Y$ associated to
$\Sigma \smallsetminus \{\alpha_6, \alpha_7\}$, then the set of
colors $\{D_{\alpha_1}^-, D_{\alpha_2}^-, D_{\alpha_4}^-\}$
is distinguished, and the corresponding
quotient is a parabolic induction of the wonderful variety $Z$ of the cases 2.4 and 5.5.
 
If we consider the wonderful subvariety of $Y$ associated to
$\Sigma \smallsetminus \{\alpha_2, \alpha_3, \alpha_6\}$, then the set of
colors $\{D_{\alpha_4}^-, D_{\alpha_7}^-\}$
is distinguished, and we get the case 2.5.

If we consider the wonderful subvariety of $Y$ associated to
$\Sigma \smallsetminus \{\alpha_1\}$, then the set of
colors $\{D_{\alpha_2}^+, D_{\alpha_2}^-, D_{\alpha_4}^-, D_{\alpha_7}^-\}$
is distinguished, and we get the cases 6.4 and 8.3.

Let now $Y$ be the comodel wonderful variety of cotype $\sfD_{2n}$,
which is the wonderful variety with
the following spherical system for a group of semisimple type $\sfA_{n-1} \times \sfD_n$
\[\begin{picture}(17550,4500)(-300,-2100)
\put(0,0){\usebox{\edge}}
\put(1800,0){\usebox{\susp}}
\put(5400,0){\usebox{\edge}}
\put(600,0){
\put(9300,0){\usebox{\edge}}
\put(11100,0){\usebox{\susp}}
\put(14700,0){\usebox{\bifurc}}
}
\multiput(0,0)(1800,0){2}{\usebox{\aone}}
\multiput(5400,0)(1800,0){2}{\usebox{\aone}}
\put(600,0){
\multiput(9300,0)(1800,0){2}{\usebox{\aone}}
\put(14700,0){\usebox{\aone}}
\multiput(15900,-1200)(0,2400){2}{\usebox{\aone}}
}
\put(7200,-2100){\line(0,1){1200}}
\put(7200,-2100){\line(1,0){8100}}
\put(15300,-2100){\line(0,1){1200}}
\put(5400,-900){\line(0,-1){900}}
\put(5400,-1800){\line(1,0){1700}}
\put(7300,-1800){\line(1,0){6200}}
\multiput(13500,-1800)(0,300){3}{\line(0,1){150}}
\multiput(3600,-1500)(0,300){3}{\line(0,1){150}}
\put(3600,-1500){\line(1,0){1700}}
\put(5500,-1500){\line(1,0){1600}}
\put(7300,-1500){\line(1,0){4400}}
\put(11700,-1500){\line(0,1){600}}
\multiput(1800,-900)(8100,0){2}{\line(0,-1){300}}
\put(1800,-1200){\line(1,0){1700}}
\multiput(3700,-1200)(1800,0){2}{\line(1,0){1600}}
\put(7300,-1200){\line(1,0){2600}}
\put(7200,2400){\line(0,-1){1500}}
\put(7200,2400){\line(1,0){10050}}
\put(17250,2400){\line(0,-1){3000}}
\multiput(17250,-600)(0,2400){2}{\line(-1,0){450}}
\multiput(5400,2100)(9900,0){2}{\line(0,-1){1200}}
\put(5400,2100){\line(1,0){1700}}
\put(7300,2100){\line(1,0){8000}}
\multiput(1800,1500)(9900,0){2}{\line(0,-1){600}}
\put(1800,1500){\line(1,0){3500}}
\put(5500,1500){\line(1,0){1600}}
\put(7300,1500){\line(1,0){4400}}
\multiput(0,1200)(9900,0){2}{\line(0,-1){300}}
\put(0,1200){\line(1,0){1700}}
\put(1900,1200){\line(1,0){3400}}
\put(5500,1200){\line(1,0){1600}}
\put(7300,1200){\line(1,0){2600}}
\multiput(0,600)(1800,0){2}{\usebox{\toe}}
\put(5400,600){\usebox{\toe}}
\put(11700,600){\usebox{\tow}}
\put(15300,600){\usebox{\tow}}
\put(16500,-600){\usebox{\tonw}}
\put(16500,1800){\usebox{\tosw}}
\end{picture}\]
and consider the wonderful subvariety of $Y$ associated to
$\Sigma \smallsetminus \{\alpha_2, \ldots, \alpha_{n-1}\}$. Then the set of
colors $\{D_{\alpha'_i}^-\,:\,2\leq i\leq n\}\cup\{D_{\alpha'_i}^+\,:\,3\leq i\leq n-1\}$
is distinguished, and the corresponding
quotient is a parabolic induction of the wonderful variety $Z$ of the cases 6.5 and 9.5 (respectively obtained for $n=4$ and $n=6$).

We are left with the case 10.5. Here the variety $Z$ is equal to the case $\mathsf a^\mathsf y(1,1)+\mathsf b'(1)$ for which the surjectivity holds thanks to \cite[Proposition~2.12]{BCG}.

\subsubsection{Basic cases} In the cases 1.3 and 10.4 the wonderful variety $Z$ has the spherical system $(\textbf{A})$. 

In the cases 3.7, 3.8 and 3.9 the wonderful variety $Z$ is a wonderful subvariety of the wonderful variety with spherical system $(\textbf{B})$.

In the cases 7.11 and 7.12 the wonderful variety $Z$ has the spherical system $(\textbf{C})$.

In the case 9.4 the wonderful variety $Z$ has the spherical system $(\textbf{D})$.\\

In the following subsections we study the surjectivity of the multiplication in the four basic cases (\textbf{A}), (\textbf{B}), (\textbf{C}), (\textbf{D}). In all these cases, we will denote by $\grS = \{\grs_1, \grs_2, \ldots \}$ the corresponding set of spherical roots, and and by $\grD = \{D_1, D_2, \ldots\}$ the corresponding set of colors.

We keep the notation of \cite[Section 2]{BCG} concerning the combinatorics of colors, and we refer to the same paper for some general background on the multiplication as well, see especially (Section 2.1 therein). In particular, we will freely make use of the partial order $\leq_\grS$ on $\mN\grD$, of the notions of covering difference and height, and of the notions of fundamental triple and of low triple.

\subsection{Case A} \label{ss:CaseA}
\[\begin{picture}(4200,1800)(-300,-900)
\put(0,0){\usebox{\dynkincthree}}
\put(0,0){\usebox{\aprime}}
\put(1800,0){\usebox{\aprime}}
\put(3600,0){\usebox{\aone}}
\end{picture}\] 
Enumerate the spherical roots as $\grs_1 = 2\gra_1$, $\grs_2 = 2\gra_2$, $\grs_3 = \gra_3$, and enumerate the colors as $D_1 = D_{\gra_1}$, $D_2 = D_{\gra_2}$, $D_3 = D_{\gra_3}^+$, $D_4 = D_{\gra_3}^-$. Then the spherical roots are expressed in terms of colors as follows:
\[\begin{array}{rr}
\grs_1 = & 2D_1 - D_2, \\ 
\grs_2 = & -D_1 +2D_2 -2D_4, \\
\grs_3 = & -D_2 + D_3 +D_4
\end{array}\]

\begin{lemma}
Let $\grg \in \mN\grS$ be a covering difference, then either $\grg \in \grS$ or $\grg$ is one of the following:
\begin{itemize}
	\item[-] $\grg_4 = \grs_1+\grs_2 = D_1 +D_2 -2D_4$;
	\item[-] $\grg_5 = \grs_2+\grs_3 = -D_1 + D_2 + D_3 -D_4$;
	\item[-] $\grg_6 = \grs_2+2\grs_3 = -D_1 + 2D_3$;
	\item[-] $\grg_7 = \grs_1+\grs_2+\grs_3 = D_1 + D_3 -D_4$.
\end{itemize}
In particular, $\height(\grg^+) = 2$.
\end{lemma}

\begin{proof}
Denote $\grg_i = \grs_i$ for all $i \leq 3$. Notice that $\grg_i$ is a covering difference for all $i \leq 7$: namely, $\grg_i^- <_\grS \grg_i^+$ and $\grg_i^-$ is maximal with this property. Suppose now that $\grg \in \mN\grS$ is a covering difference and assume that $\grg \neq \grg_i$ for all $i$. Notice that $\grg$ cannot be a nontrivial multiple of any other covering difference. Write $\grg = a_1 \grs_1 + a_2 \grs_2 + a_3 \grs_3 = c_1 D_1 + c_2 D_2 +c_3 D_3 +c_4 D_4$, then
\[\begin{array}{rr}
	c_1 = & 2a_1 - a_2,\\
	c_2 = & -a_1+ 2a_2 -a_3, \\
	c_3 = & a_3, \\
	c_4 = & -2a_2 +a_3.
\end{array}\]

Suppose that $a_3 = 0$. Since $\grg$ cannot be a multiple of a covering difference, it follows $a_1 > 0$ and $a_2 > 0$. Since $\grg^+ - \grg_i \not \in \mN \grD$ for $i =1,2,4$, we have $c_1 + c_2 \leq 1$. On the other hand $c_1 + c_2 = a_1 + a_2 \geq 2$, absurd.

Suppose that $a_3 > 0$. Then $c_3 > 0$. If both $a_1 > 0$ and $a_2 > 0$, then $\grg^+ - \grg_i \not \in \mN\grD$ and $\grg^- + \grg_i \not \in \mN\grD$ for all $i=1,2,3,4,5,7$, and it easily follows $c_2 = c_4 = 0$, hence $a_1 = -c_2 -c_4 = 0$, absurd. We must have $a_1 = 0$: indeed otherwise $a_2=0$ implies $c_1 = 2a_1$, hence $\grg^+ - \grs_1 \in \mN\grD$. Therefore $c_1 <0$ and $c_3 > 0$. Since $\grg^+ - \grs_3 \not \in \mN\grD$ and $\grg^- + \grg_5 \not \in \mN\grD$, it follows $c_4 = 0$, hence $\grg$ is a multiple of $\grg_6$, absurd.
\end{proof}

Since every covering difference $\grg \in \mN \grS$ satisfies $\height(\grg^+)=2$, it follows that every fundamental triple is low. In particular we get the following classification of the low fundamental triples.

\begin{lemma}
Let $(D,E,F)$ be a low fundamental triple, denote $\grg = D+E-F$ and suppose that $\grg \neq 0$. Then, up to switching $D$ and $E$, the triple $(D,E,F)$ is one of the following: 
\begin{itemize}
	\item[-] $(D_1,D_1, D_2)$, $\grg = \grs_1$;
	\item[-] $(D_1,D_2, 2D_4)$, $\grg = \grs_1 + \grs_2$;
	\item[-] $(D_1,D_3, D_4)$, $\grg = \grs_1 + \grs_2 + \grs_3$;
	\item[-] $(D_2,D_2, D_1+2D_4)$, $\grg = \grs_2$;
	\item[-] $(D_2,D_3, D_1+D_4)$, $\grg = \grs_2+\grs_3$;
	\item[-] $(D_3,D_3, D_1)$, $\grg = \grs_2+2\grs_3$;
	\item[-] $(D_3,D_4, D_2)$, $\grg = \grs_3$.
\end{itemize}
\end{lemma}

\begin{proposition}
The multiplication $m_{D,E}$ is surjective for all $D,E\in\mathbb N\Delta$. 
\end{proposition}

\begin{proof}
It is enough to show that $s^{D+E-F}V_F\subset V_D\cdot V_E$ for all low fundamental triples. Moreover, notice that in this case the surjectivity of the multiplication is already known for all proper wonderful subvarieties. Indeed, if we remove the spherical root $\sigma_3$ we have a parabolic induction of a wonderful symmetric variety, if we remove the spherical root $\sigma_1$ we have a parabolic induction of a wonderful subvariety of $\mathsf a^\mathsf y(1,1)+\mathsf b'(1)$, if we remove $\sigma_2$ we have a parabolic induction of the direct product of two rank one wonderful varieties. Therefore, we are left to the only low fundamental triple with $\mathrm{supp}_\Sigma(D+E-F)=\{\sigma_1,\sigma_2,\sigma_3\}$, namely $(D_1,D_3,D_4)$.

Let us consider the symmetric pair $(\mathfrak g=\mathfrak f_4,\mathfrak k=\mathfrak c_3+\mathfrak a_1)$, number 10 in our list, we have $\mathfrak p=V(\omega_3+\omega')$, and the $\mathfrak k$-action on $\mathfrak p$ gives a map $\varphi\colon\mathfrak k\otimes\mathfrak p\to\mathfrak p$. Restricting the map to the tensor product of $\mathfrak c_3\subset\mathfrak k$ with the simple $\mathfrak c_3$-submodule containing the highest weight vector $V(\omega_3)\subset\mathfrak p$, we get a map $\varphi\colon\mathfrak c_3\otimes V(\omega_3)\to V(\omega_3)$, hence a map $\varphi\colon V_{D_1}\otimes V_{D_3} \to V_{D_4}$.

Let us fix $h_{D_3}=e$ as in the case~10.4 of the list in Appendix~\ref{A}, and take $\mathrm N_K(K_e)$. Recall its description given in Section~\ref{s:1}, cases~1.3 and~10.4. It follows that $h_{D_1}$ belongs to the 1-dimensional $L_{[e]}$-submodule of $P^\mathrm u$ ($\mathfrak u$ in the notation of Appendix~\ref{A}, case~10.4). By construction, we have $[\mathfrak u,e]\neq0$, that is, $\varphi(h_{D_1}\otimes h_{D_3})\neq 0$. 
\end{proof}

\subsection{Case B} \label{ss:CaseB}

\[\begin{picture}(3600,3000)(-300,-1500)
\put(0,0){\usebox{\dynkindfour}}
\put(0,0){\usebox{\athreene}}
\put(0,0){\usebox{\athreese}}
\put(1800,0){\usebox{\athreebifurc}}
\end{picture}\]

Enumerate the spherical roots as $\grs_1 = \gra_1 + \gra_2 + \gra_3$, $\grs_2 = \gra_1 + \gra_2 + \gra_4$, $\grs_3 = \gra_2 + \gra_3 + \gra_4$, and enumerate the colors as $D_1 = D_{\gra_4}$, $D_2 = D_{\gra_3}$, $D_3 = D_{\gra_1}$. Then the spherical roots are expressed in terms of colors as follows: 
\[\begin{array}{rr}
\grs_1 = & -D_1 + D_2 +D_3, \\ 
\grs_2 = & D_1 - D_2 +D_3, \\
\grs_3 = & D_1 + D_2 -D_3.
\end{array}\]

It is immediate to see that every covering difference $\grg \in \mN \grS$ is either a spherical root or the sum of two spherical roots, and it satisfies $\height(\grg^+) = 2$. As well, one can easily get the following description of the low fundamental triples.

\begin{lemma}\label{lem:triple}
Let $(D,E,F)$ be a low fundamental triple, denote $\grg = D+E-F$ and suppose that $\grg \neq 0$. Then, up to switching $D$ and $E$, the triple $(D,E,F)$ is one of the following: 
\begin{itemize}
	\item[-] $(D_1,D_2, D_3)$, $\grg = \grs_3$;
	\item[-] $(D_2,D_3, D_1)$, $\grg = \grs_1$;
	\item[-] $(D_3,D_1, D_2)$, $\grg = \grs_2$;
	\item[-] $(D_1,D_1, 0)$, $\grg = \grs_2+\grs_3$;
	\item[-] $(D_2,D_2, 0)$, $\grg = \grs_1+\grs_3$;
	\item[-] $(D_3,D_3, 0)$, $\grg = \grs_1+\grs_2$.
\end{itemize}
\end{lemma}

\begin{proposition}
The multiplication $m_{D,E}$ is surjective for all $D,E\in\mathbb N\Delta$. 
\end{proposition}

\begin{proof}
It is enough to show that $s^{D+E-F}V_F\subset V_D\cdot V_E$ for all low fundamental triples. For the wonderful subvarieties of rank 1 the surjectivity of the multiplication is known. Thus we are left with the low fundamental triples $(D,E,F)$ where $D+E-F$ is the sum of two spherical roots. By symmetry, it is enough to consider $(D_3,D_3,0)$. The subset  $\{D_1, D_2\}\subset\Delta$ is distinguished, and the quotient is a wonderful symmetric variety (of rank 1), whose multiplication is surjective.     
\end{proof}

\subsection{Case C} \label{ss:CaseC}

\[\begin{picture}(7800,3000)(-300,-2100)
\put(0,0){\usebox{\dynkinesix}}
\put(0,0){\usebox{\afour}}
\multiput(0,0)(3600,-1800){2}{\circle{600}}
\multiput(0,0)(25,-25){13}{\circle*{70}}
\put(300,-300){\multiput(0,0)(300,0){10}{\multiput(0,0)(25,25){7}{\circle*{70}}}\multiput(150,150)(300,0){10}{\multiput(0,0)(25,-25){7}{\circle*{70}}}}
\multiput(3600,-1800)(-25,25){13}{\circle*{70}}
\put(3300,-1500){\multiput(0,0)(0,300){4}{\multiput(0,0)(25,25){7}{\circle*{70}}}\multiput(150,150)(0,300){4}{\multiput(0,0)(-25,25){7}{\circle*{70}}}}
\put(7200,0){\usebox{\aone}}
\end{picture}\]

Enumerate the spherical roots as $\grs_1 = \gra_1 + \gra_3 + \gra_4 + \gra_5$, $\grs_2 = \gra_1 +\gra_2 + \gra_3 + \gra_4$, $\grs_3 = \gra_6$, and enumerate the colors as $D_1 = D_{\gra_1}$, $D_2 = D_{\gra_2}$, $D_3 = D_{\gra_5}$, $D_4 = D_{\gra_6}^+$, $D_5 = D_{\gra_6}^-$. Then the spherical roots are expressed in terms of colors as follows: 
\[\begin{array}{rr}
\grs_1 = & D_1 - D_2 + D_3 - D_5, \\ 
\grs_2 = & D_1 +D_2 -D_3, \\
\grs_3 = & -D_3 + D_4 +D_5.
\end{array}\]

\begin{lemma}
Let $\grg \in \mN\grS$ be a covering difference, then either $\grg \in \grS$ or $\grg$ is one of the following:
\begin{itemize}
	\item[-] $\grg_4 = \grs_1+\grs_2 = 2D_1 -D_5$;
	\item[-] $\grg_5 = \grs_1+\grs_3 = D_1 -D_2 +D_4$.
\end{itemize}
In particular, $\height(\grg^+) = 2$.
\end{lemma}

\begin{proof}
Denote $\grg_i = \grs_i$ for all $i \leq 3$. Notice that $\grg_i$ is a covering difference for all $i \leq 5$: namely, $\grg_i^- <_\grS \grg_i^+$ and $\grg_i^-$ is maximal with this property. Suppose now that $\grg \in \mN\grS$ is a covering difference and assume that $\grg \neq \grg_i$ for all $i$. Write $\grg = a_1 \grs_1 + a_2 \grs_2 + a_3 \grs_3 = c_1 D_1 + c_2 D_2 +c_3 D_3 +c_4 D_4 +c_5 D_5$, then
\[\begin{array}{rr}
	c_1 = & a_1 +a_2,\\
	c_2 = & -a_1 +a_2, \\
	c_3 = & a_1 -a_2 -a_3, \\
	c_4 = & a_3, \\
	c_5 = & -a_1 +a_3.
\end{array}\]

Notice that $\grg$ cannot be a nontrivial multiple of any other covering difference, thus $c_1 > 0$. 

Suppose that $a_1 >0$, $a_2 > 0$ and $a_3 > 0$. Then none of $\grg^+ - \grg_i$ and $\grg^- + \grg_i$ is in $\mN \grD$, for every $i \leq 5$. It easily follows that $c_1=1$ and $c_2 = 0$, which is absurd because $c_1 + c_2 = 2 a_2$.

To conclude the proof it is enough to show that, if $a_i = 0$ for some $i$, then $\grg$ is a multiple of some $\grg_i$. Suppose that $a_1 = 0$: then $\grg^+ - \grs_2 \not \in \mN \grD$, hence $c_2 \leq 0$ and it follows $a_2 = 0$. Suppose that $a_2 = 0$: then $\grg^+ - \grg_5 \not \in \mN \grD$, hence $c_4 \leq 0$, and it follows $a_3 = 0$. Suppose that $a_3 = 0$: then $\grg^- + \grg_4 \not \in \mN \grD$, hence $c_5 \geq 0$ and it follows $a_1 = 0$. 
\end{proof}

Since every covering difference $\grg \in \mN \grS$ satisfies $\height(\grg^+)=2$, it follows that every fundamental triple is low. In particular we get the following classification of the low fundamental triples.

\begin{lemma}
Let $(D,E,F)$ be a low fundamental triple, denote $\grg = D+E-F$ and suppose that $\grg \neq 0$. Then, up to switching $D$ and $E$, the triple $(D,E,F)$ is one of the following: 
\begin{itemize}
	\item[-] $(D_1, D_1, D_5)$, $\grg = \grs_1+\grs_2$;
	\item[-] $(D_1, D_2, D_3)$, $\grg = \grs_2$;
	\item[-] $(D_1, D_3, D_2 +D_5)$, $\grg = \grs_1$;
	\item[-] $(D_1, D_4, D_2)$, $\grg = \grs_1+\grs_3$;
	\item[-] $(D_4, D_5, D_3)$, $\grg = \grs_3$.
\end{itemize}
\end{lemma}

\begin{proposition}
The multiplication $m_{D,E}$ is surjective for all $D,E\in\mathbb N\Delta$. 
\end{proposition}

\begin{proof}
It is enough to show that $s^{D+E-F}V_F\subset V_D\cdot V_E$ for all low fundamental triples. For the wonderful subvarieties of rank 1 the surjectivity of the multiplication is known. We are left with the low fundamental triples $(D,E,F)$ where $D+E-F$ is the sum of two spherical roots, that is, $(D_1,D_1,D_5)$ and $(D_1,D_4,D_2)$.

To treat $(D_1,D_1,D_5)$ we can consider the wonderful subvariety with spherical roots $\sigma_1$ and $\sigma_2$. Here $\{D_2,D_3\}$ is a distinguished subset of colors, giving a wonderful symmetric variety (of rank 1) as quotient, whose multiplication is surjective.

To treat $(D_1,D_4,D_2)$ we consider the wonderful subvariety with spherical roots $\sigma_1$ and $\sigma_3$. Now $\{D_3,D_5\}$ is a distinguished subset of colors, giving again a wonderful symmetric variety (of rank 1) as quotient, with surjective multiplication.     
\end{proof}

\subsection{Case D}

\[\begin{picture}(9600,2700)(-300,-1800)
\put(0,0){\usebox{\dynkineseven}}
\multiput(0,0)(7200,0){2}{\usebox{\gcircle}}
\put(9000,0){\usebox{\aone}}
\end{picture}\] 

Enumerate the spherical roots as $\grs_1 = 2\gra_1 + \gra_2 + 2\gra_3 + 2\gra_4 + \gra_5$, $\grs_2 = \gra_2 + \gra_3 + 2\gra_4 + 2\gra_5 + 2\gra_6$, $\grs_3 = \gra_7$, and enumerate the colors as $D_1 = D_{\gra_1}$, $D_2 = D_{\gra_6}$, $D_3 = D_{\gra_7}^+$, $D_4 = D_{\gra_7}^-$. Then the spherical roots are expressed in terms of colors as follows: 
\[\begin{array}{rr}
\grs_1 = & 2D_1 - D_2, \\ 
\grs_2 = & -D_1 +2D_2 -2D_4, \\
\grs_3 = & -D_2 + D_3 +D_4.
\end{array}\]

Notice that same Cartan matrix is the same of case (\textbf A). In particular, the description of the covering differences and of the low fundamental triples is the same.

\begin{proposition}
The multiplication $m_{D,E}$ is surjective for all $D,E\in\mathbb N\Delta$. 
\end{proposition}

\begin{proof}
The proof is similar to that in case (\textbf A). We have to show that $s^{D+E-F}V_F\subset V_D\cdot V_E$ for all low fundamental triples. In this case as well, the surjectivity of the multiplication is known for all proper wonderful subvarieties. Indeed, if we remove the spherical root $\sigma_3$ we have a parabolic induction of a wonderful symmetric variety, if we remove the spherical root $\sigma_1$ we have a parabolic induction of a wonderful subvariety of $\mathsf a^\mathsf y(1,1)+\mathsf d(5)$ (see \cite[\S2.2]{BCG}), and if we remove $\sigma_2$ we have a parabolic induction of a direct product of two rank 1 wonderful varieties. Therefore, we are left with the only low fundamental triple with $\mathrm{supp}_\Sigma(D+E-F)=\{\sigma_1,\sigma_2,\sigma_3\}$, namely $(D_1,D_3,D_4)$.

Let us consider the symmetric pair $(\mathfrak g=\mathfrak e_8,\mathfrak k=\mathfrak e_7+\mathfrak a_1)$, number 9 in our list. We have $\mathfrak p=V(\omega_7+\omega')$, and the $\mathfrak k$-action on $\mathfrak p$ gives a map $\varphi\colon\mathfrak k\otimes\mathfrak p\to\mathfrak p$. Restricting the map to the tensor product of $\mathfrak e_7\subset\mathfrak k$ with the simple $\mathfrak e_7$-submodule containing the highest weight vector $V(\omega_7)\subset\mathfrak p$, we get a map $\varphi\colon\mathfrak e_7\otimes V(\omega_7)\to V(\omega_7)$, hence a map $\varphi\colon V_{D_1}\otimes V_{D_3} \to V_{D_4}$.

Let us fix $h_{D_3}=e$ as in the case~9.4 of the list in Appendix~\ref{A}, and take $\mathrm N_K(K_e)$. Recall its description given in Section~\ref{s:1}, case~9.4. It follows that $h_{D_1}$ lies in the 1-dimensional $L_{[e]}$-submodule of $P^\mathrm u$ ($\mathfrak u$ in the notation of Appendix~\ref{A}, case~9.4). By construction, we have $[\mathfrak u,e]\neq0$, that is, $\varphi(h_{D_1}\otimes h_{D_3})\neq 0$. 
\end{proof}

%%%%%%%%%%%%%%%%%%%%%%%%%%%%%%%%%%%%%%%%
%%%%%%%%%%%%%%%%%%%%%%%%%%%%%%%%%%%%%%%%%%%%%%%%%%%%%%%%%%%%%%%%%%%%%%%%%%%%%%%

\section{Normality and semigroups}\label{s:3}

%Let $G$ be a simple group of exceptional type and let $K \subset G$ be a symmetric subgroup and denote by $\gop = \bigoplus_{i=1}^N$ the decomposition of $\gop$ into irreducible $K$-modules. As is well known, $\gop$ is reducible if and only if $(G,K)$ is of Hermitian type (that is the center $Z(K)$ has positive dimension), in which case $N=2$ and $\gop_1 \simeq \gop_2^*$. 
%
%More precisely, as in \cite{BCG} and \cite{BG}, if $v = \sum_{i=1}^N v_i$ with $v_i \in \gop_i$ and if $[v_i] \in \mP(\gop_i)$ is the line defined by $v_i$, then setting $\pi(v) =  ([v_1], \ldots, [v_N])$ defines a morphism
%$$\pi : Ke \ra \mP(\gop_1) \times \ldots \times \mP(\gop_N).$$
%Then the spherical orbit $K\pi(e)$ admits a wonderful compactification $X$, and $\ol{Ke}$ is the multicone over $\ol{K\pi(e)}$ associated to the decomposition of $\gop$ into irreducible $K$-module.

Let $e \in \gop$ be a nilpotent element and suppose that $Ke$ is spherical, we study in this section the normality of the closure $\ol{Ke}$, as well as the $K$-module structure of the coordinate ring $\mC[\wt{Ke}]$ of its normalization. This reduces to a combinatorial problem on the wonderful $K$-variety $X$ that we associated to $Ke$. We denote by $\grS$ the set of spherical roots of $X$, and by $\grD$ its set of colors.

We keep the notation introduced in the previous section. By its very definition, $X$ is the wonderful compactification of $K/K_{\pi(e)}$. Thus, by the theory of spherical embeddings, $X$ is endowed with a $K$-equivariant morphism $\phi_i : X \ra \mP(\gop_i)$, for all $i=1, \ldots, M$. For all such $i$, let $D_{\gop_i} \in \mN\grD$ be the unique $B$-stable divisor such that $\calL_{D_{\gop_i}} = \phi_i^* \calO(1)_{\mP(\gop_i)}$. It follows that $\gop_i^* \simeq V_{D_{\gop_i}}$ is naturally identified with the submodule of $\grG(X,\calL_{D_{\gop_i}})$ generated by the canonical section of $D_{\gop_i}$, which is a highest weight section in
$$
		\grG(X,\calL_{D_{\gop_i}}) = \bigoplus_{D \leq_\grS D_{\gop_i}} s^{D_{\gop_i} - D} \, V_D  
$$

Since $\ol{Ke}$ is the multicone over $\ol{K\pi(e)}$, by \cite[Theorem 1.2]{BG} it follows that $\ol{Ke}$ is normal if and only if $D_{\gop_i}$ is a minuscule element in $\mN \grD$ for all $i=1, \ldots, M$. By making use of such criterion, we prove the following theorem.

\begin{theorem}	\label{teo:normal}
Let $(\gog, \gok)$ be a symmetric pair with $\gog$ of exceptional type and let $\calO \subset \gop$ be a spherical nilpotent orbit. Then $\ol \calO$ is not normal if and only if $(\gog,\gok) = (\sfG_2, \sfA_1 \times \sfA_1)$, and the Kostant-Dynkin diagram of $\calO$ is $(1;3)$. 
\end{theorem}

\begin{remark}
When $G$ is exceptional and $\height(e) \leq 3$, then $\ol{Ge}$ is not normal if and only if $G$ is of type $\sfG_2$ and the Kostant-Dynkin diagram of $Ge$ is $(10)$ (see \cite[Table 2]{Pa} for an account on these results). The remaining cases with $G$ exceptional and $Ke$ spherical only occur for $G$ of type $\sfE_6$, $\sfE_7$, $\sfF_4$, and specifically for the nilpotent orbits of Cases 3.6, 3.7, 3.8, 3.9, 7.10, 7.11, 7.12, 11.2. When $G$ is of type $\sfE_6$ or $\sfF_4$, the normal nilpotent varieties have been classified respectively by Sommers \cite{sommers} and by Broer \cite{broer}. In particular, we have that $\ol{Ge}$ is normal in Case 3.6, and not normal in Cases 3.7, 3.8, 3.9, 11.2. When $G$ is of type $\sfE_7$ (and $\sfE_8$) the classification of the normal nilpotent varieties is still not complete in the literature, but it seems that the $G$-orbit closures of Cases 7.10, 7.11, 7.12 are {\em expected} to be normal (see \cite[\S 1.9.3]{FJLS} and \cite[Section 7.8]{broer2}).
\end{remark}

For the coordinate ring of the normalization $\wt{Ke}$, by \cite[Theorem 1.3]{BG} we have the following description:
$$
	\mC[\wt{Ke}] = \bigoplus_{n_1,\ldots,n_M\geq0} \grG(X,\calL^{n_1}_{D_{\gop_1}}\otimes\ldots\otimes\calL^{n_M}_{D_{\gop_M}}).
$$
Denote $\grD_\gop(e) = \{D_{\gop_1},\ldots, D_{\gop_M}\}$. It follows by the previous description that, to compute $\grG(\wt{Ke})$, the weight semigroup of $\wt{Ke}$, is enough to compute the semigroup 
$$
	\grG_{\grD_\gop(e)} = \{(n_1,\ldots,n_M,D)\in\mN^M\times\mN \grD  \; : \; D \leq_\grS n_1 D_{\gop_1} + \ldots + n_M D_{\gop_M}\}.
$$
Indeed, by definition $\grG(\wt{Ke})$ consists of the weights
$$
	 n_1\lambda_1^*+\ldots+n_M\lambda_M^*-(n_1D_{\gop_1}+\ldots+n_MD_{\gop_M}-D)
$$
for $n_1,\ldots,n_M\geq0$ and $D \leq_\grS n_1 D_{\gop_1} + \ldots + n_M D_{\gop_M}$, where $\lambda_1^*,\ldots,\lambda_M^*$ are the highest weights of the $K$-modules $\gop_1^*,\ldots,\gop_M^*$.

Suppose that $(G,K)$ is not of Hermitian type. As already recalled, in this case $\gop$ is an irreducible $K$-module, thus $\grD_\gop(e)$ consists of a unique element, that we denote by $D_\gop$. Moreover, since $K$ is semisimple and $Ke$ is spherical, in this case is enough to compute the semigroup
$$
	\grG_{D_\gop} = \{D\in\mN \grD  \; : \; D \leq_\grS nD_{\gop}\text{ for some }n\in\mN\}.
$$
Indeed, $\grG(\wt{Ke})$ is the image of $\grG_{D_\gop}$ via the canonical homomorphism $\gro\colon \mZ \grD \ra \calX(B)$ induced by the restriction of line bundles to the closed orbit of $X$ (see \cite[Section 2]{BCG} for the combinatorial description of $\omega$).

Suppose now that $(G,K)$ is of Hermitian type. Then $K$ is the Levi subgroup of a parabolic subgroup of $G$ with Abelian unipotent radical, thus $K$ is a maximal Levi subgroup of $G$ and the identity component $Z_K$ of its center is one dimensional. By \cite[Proposition 4.1]{BG}, $Z_K$ acts on $\gop_1$ and $\gop_2$ via nontrivial opposite characters in $\calX(Z_K) \simeq \mZ$, which can be described as follows.

Let us fix $T\subset B\subset G$ a maximal torus and a Borel subgroup of $G$, and $Q$ a standard parabolic subgroup of $G$ with Abelian unipotent radical, such that $K$ is the standard Levi subgroup of $Q$. Let us fix $K\cap B$ as Borel subgroup of $K$. We denote by $\gop_+$ the nilpotent radical of $\mathrm{Lie}\,Q$ and by $\gop_-$ the nilpotent radical of $\mathrm{Lie}\,Q_-$, so that $\gop=\gop_+\oplus\gop_-$. Let $S = \{\gra_1, \ldots, \gra_n\}$ be the set of simple roots of $G$ and denote by $\theta_G$ the highest root of $G$. Then $\theta_G$ is the highest weight of the simple $K$-module $\gop_+$.

Let $\chi \in \calX(Z_K)$  be the character given by the action on $\gop_+$. To describe $\chi$, recall that a standard parabolic subgroup of $G$ has Abelian unipotent radical if and only if it is maximal, and the corresponding simple root $\gra_p$ has coefficient 1 in $\theta_G$. In particular, when $G$ is a simple group of exceptional type, we have the following possibilities:
\begin{itemize}
	\item[(1)] If $G$ is of type $\sfE_6$: $\gra_1, \gra_6$.
	\item[(2)] If $G$ is of type $\sfE_7$: $\gra_7$.
\end{itemize}
Let $\got_G \subset \gog$ be the Cartan subalgebra generated by the fundamental coweights $\gro_1^\vee, \ldots, \gro_n^\vee$, and let $\got_K^\mss \subset \got_G$ be the subalgebra generated by the simple coroots of $K$. Take $K$ to be the maximal standard Levi subgroup of $G$ corresponding to $\gra_p$, then the Lie algebra $\goz_K = \Lie Z_K$ is generated by the fundamental coweight $\gro_p^\vee$.

Assuming $G$ to be simply connected, we have that $\calX(T)^\vee = \mZ S^\vee$, therefore
$$\calX(Z_K)^\vee = \goz_K \cap \mZ S^\vee
$$ 
is generated by $m \gro_p^\vee$, where $m \in \mathbb N$ is the minimum such that $m \gro_p^\vee \in \mZ S^\vee$. If $z(\xi)$ is the 1-parameter-subgroup of $Z_K$ corresponding to $m\omega_p^\vee$, we have
$$
	\chi(z(\xi)) = \xi^{m\theta_G(\omega_p^\vee)} = \xi^m.
$$
In our cases, we have the following possibilities for the value of $m$, depending on the pair $(G,\gra_p)$.
\begin{itemize}
	\item[-] $(\sfE_6, \gra_1)$, $(\sfE_6, \gra_6)$: $m=3$.
	\item[-] $(\sfE_7, \gra_7)$: $m = 2$.
\end{itemize}

In the following we compute the weight semigroup of $\wt{Ke}$. We omit the cases when $X$ has rank smaller than three, or is obtained by parabolic induction from a symmetric variety: in all these cases the computations are fairly easy. At last, despite it is of rank one, we give some details for the unique nonnormal case, which is in type $\mathsf G_2$.

\subsection{Cases 1.3, 2.4, 5.5, 6.4, 8.3, 9.4, 10.4} We consider the case 1.3, the others are treated essentially in the same way: indeed, apart from colors which take nonpositive values against every spherical root, they all have the Cartan pairing described in \S\ref{ss:CaseA}.

Keeping a similar notation, we have $\grS = \{\grs_1, \grs_2, \grs_3\}$ and $\grD = \{D_1, \ldots, D_5\}$, where we denote $\grs_1 = 2\gra_2$, $\grs_2 = 2\gra_3$, $\grs_3 = \gra_4$, and $D_1 = D_{\gra_2}$, $D_2 = D_{\gra_3}$, $D_3 = D_{\gra_4}^+$, $D_4 = D_{\gra_4}^-$, $D_5 = D_{\gra_1}$. In particular, the equations relating spherical roots and colors are all the same as in \S\ref{ss:CaseA}, but the first one which reads
$$
	\grs_1 = 2D_1 - D_2 -2D_5.
$$

Notice that $D_\gop = D_3$ is minuscule in $\mN\grD$, in particular $\ol{Ke}$ is normal. On the other hand we have
\begin{align*}
&	D_1 = 2D_3- (\grs_2 + 2\grs_3)\\
& D_4+2D_5 = 3D_3 - (\grs_1+2\grs_2 + 3\grs_3)\\
& D_2+ 2D_5 = 4D_3 - (\grs_1+2\grs_2 + 4\grs_3)
\end{align*}
On the other hand if $D \in \grG_{D_3}$ and $D = \sum_{i=1}^5 c_i D_i$, then $c_5 = 2(c_2+c_4)$. Therefore $\grG_{D_3}$ is freely generated by $D_3$, $D_1$, $D_4+2D_5$, $D_2+2D_5$.

\subsection{Cases 2.5, 6.5, 9.5, 10.5} We consider the case 2.5, the others are treated essentially in the same way since they have the same Cartan pairing, apart from possible colors taking nonpositive values against every spherical root. We have in this case $\grS = \{\grs_1, \grs_2, \grs_3\}$ and $\grD = \{D_1, \ldots, D_6\}$, where we denote $\grs_1 = \gra'$, $\grs_2 = \gra_3$, $\grs_3 = \gra_2+\gra_4$, and $D_1 = D_{\gra_3}^+$, $D_2 = D_{\gra_3}^-$, $D_3 = D_{\gra_2}$, $D_4 = D_{\gra'}^-$, $D_5 = D_{\gra_5}$, $D_6 = D_{\gra_1}$. Then colors and spherical roots are related by the following equations:
\begin{align*}
&	\grs_1 = D_1 +D_2 -D_3 \\
& \grs_2 = D_1 -D_2 +D_3 -D_4 \\
& \grs_3 = -2D_3 +2D_4 -D_5 -D_6
\end{align*}

Notice that $D_\gop = D_1$ is minuscule in $\mN\grD$, in particular $\ol{Ke}$ is normal. Moreover
\begin{align*}
&	D_4 = 2D_1- (\grs_1 + \grs_2)\\
& 2D_2 +D_5 +D_6 = 2D_1 - (2\grs_2 + \grs_3)\\
& D_2 +D_3 +D_5 +D_6 = 3D_1 - (\grs_1+2\grs_2 + \grs_3)\\
& 2D_3 +D_5 +D_6 = 4D_1 - (2\grs_1+2\grs_2 + \grs_3)\\
\end{align*}
On the other hand, if $D \in \grG_{D_1}$ and $D = \sum_{i=1}^6 c_i D_i$, then $c_5 = c_6$ and $c_2+c_3 = 2c_5$, therefore $\grG_{D_1}$ is generated by $D_1$, $D_4$, $2D_2+D_5+D_6$, $D_2+D_3+D_5+D_6$, $2D_3+D_5+D_6$.

\subsection{Case 3.9}
Notice that this case is obtained by parabolic induction from the wonderful variety considered in \S\ref{ss:CaseB}. We have in this case $\grS = \{\grs_1, \grs_2, \grs_3\}$ and $\grD = \{D_1, \ldots, D_4\}$, where we denote $\grs_1 = \gra_2 + \gra_3 + \gra_4$, $\grs_2 = \gra_2 + \gra_3 + \gra_5$, $\grs_3 = \gra_3 + \gra_4 + \gra_5$, and $D_1 = D_{\gra_5}$, $D_2 = D_{\gra_4}$, $D_3 = D_{\gra_2}$, $D_4 = D_{\gra_1}$. In particular, the equations relating spherical roots and colors of \S\ref{ss:CaseB} become
\begin{align*}
&	\grs_1 = -D_1 + D_2 +D_3 -D_4,\\
&	\grs_2 =  D_1 - D_2 + D_3 -D_4,\\
&	\grs_3 =  D_1 + D_2 - D_3.
\end{align*}
Notice that $\grD_\gop(e) = \{D_1,D_2\}$, and that both $D_1$ and $D_2$ are minuscule in $\mN\grD$. In particular $\ol{Ke}$ is normal. Following \cite[Remark 4.7]{BG}, to compute the semigroup $\grG_{\grD_\gop(e)}$ it is enough to compute the semigroup
$$
	\grG^\grS_{\grD_\gop(e)} = \{ \grg \in \mN\grS \; : \; \supp(\grg^+) \subset \{D_1, D_2\}\}.
$$
Let $\grg \in \mN\grS$ and write $\grg = a_1\grs_1 + a_2 \grs_2 + a_3 \grs_3$, then $\grg \in \grG^\grS_{\grD_\gop(e)}$ if and only if $a_1+a_2 \leq a_3$. Therefore $\grG^\grS_{\grD_\gop(e)}$ is generated by the elements
$$
 D_1 + D_2 -D_3 = \grs_3, \;  2D_1 -D_4 = \grs_2 + \grs_3, \; 2D_2 -D_4 = \grs_1 + \grs_3.
$$

\subsection{Cases 5.8, 5.9} 
We treat the case 5.8, the other one is identical. Notice that this case is obtained by parabolic induction from the comodel wonderful variety of cotype $\sfE_7$: in particular the combinatorics of colors and spherical roots is essentially the same as that of a model wonderful variety. We have in this case $\grS = \{\grs_1, \ldots, \grs_6\}$ and $\grD = \{D_1, \ldots, D_8\}$, where we enumerate the spherical roots as $\grs_1 = \gra_4$, $\grs_2 = \gra_3$, $\grs_3 = \gra_1$, $\grs_4 = \gra_5$,  $\grs_5 = \gra_2$, $\grs_6 = \gra_6$, and the colors as $D_1 = D_{\gra_4}^+$, $D_2 = D_{\gra_3}^-$, $D_3 = D_{\gra_4}^-$, $D_4 = D_{\gra_1}^+$, $D_5 = D_{\gra_5}^-$, $D_6 = D_{\gra_6}^+$, $D_7 = D_{\gra_6}^-$, $D_8 = D_{\gra_7}$. Then we have the following relations between colors and spherical roots:
\begin{align*}
&	\grs_1 = D_1 +D_3 -D_4 \\
& \grs_2 = D_2 -D_3 +D_4 -D_5 \\
& \grs_3 = -D_1 -D_2 +D_3 +D_4 -D_5 \\
& \grs_4 = -D_2 -D_3 +D_4 +D_5 -D_6 \\
& \grs_5 = -D_4 +D_5 +D_6 -D_7 \\
& \grs_6 = -D_5 +D_6 +D_7 -D_8
\end{align*}

Notice that $D_\gop = D_1$, which is a minuscule element in $\mN\grD$. Therefore $\ol{Ke}$ is normal. Let $D \in \grG_{D_1}$, and write $D = nD_1 - (\sum_{i=1}^6 a_i \grs_i) = \sum_{i=1}^8 c_i D_i$. Then we have the following relations: $c_8 = c_2+c_3+c_4+c_5$ and $c_2+c_5+c_7 = 2a_3$. On the other hand, notice that
\begin{align*}
&	D_6 = 2D_1- (2\grs_1 + \grs_2 + \grs_4)\\
& 2D_7 = 3D_1 - (4\grs_1 +3\grs_2 + \grs_3 + 2\grs_4 +2\grs_5)\\
& D_3 + D_8 = 3D_1 - (3\grs_1 +2\grs_2 + 2\grs_4 + \grs_5 +\grs_6)\\
& D_4 +D_8 = 4D_1 - (4\grs_1 +2\grs_2 + 2\grs_4 + \grs_5 +\grs_6)\\
& D_2+D_7 +D_8 = 4D_1 - (5\grs_1 +3\grs_2 + \grs_3 +3\grs_4 + 2\grs_5 +\grs_6)\\
& D_5+D_7 +D_8 = 5D_1 - (6\grs_1 +4\grs_2 + \grs_3 +3\grs_4 + 2\grs_5 +\grs_6)\\
& 2D_2 + 2D_8 = 5D_1 - (6\grs_1 +3\grs_2 + \grs_3 +4\grs_4 + 2\grs_5 +2\grs_6)\\
& D_2 + D_5 +2D_8 = 6D_1 - (7\grs_1 +4\grs_2 + \grs_3 +4\grs_4 + 2\grs_5 +2\grs_6)\\
& 2D_5 +2D_8 = 7D_1 - (8\grs_1 +5\grs_2 + \grs_3 +4\grs_4 + 2\grs_5 +2\grs_6)
\end{align*}
Therefore $\grG_{D_1}$ is generated by the elements in the previous list, together with $D_1$.

\subsection{Cases 7.11, 7.12} 
We treat the case 7.12, the other one is identical. This case was already considered in \S\ref{ss:CaseC}, we keep the notation introduced therein. Notice that $\grD_\gop(e) = \{D_1,D_4\}$, and that both $D_1$ and $D_4$ are minuscule in $\mN\grD$. In particular $\ol{Ke}$ is normal. Following \cite[Remark 4.7]{BG}, to compute the semigroup $\grG_{\grD_\gop(e)}$, we only have to compute the semigroup
$$
	\grG^\grS_{\grD_\gop(e)} = \{ \grg \in \mN\grS \; : \; \supp(\grg^+) \subset \{D_1, D_2\}\}.
$$
Let $\grg \in \mN\grS$ and write $\grg = a_1\grs_1 + a_2 \grs_2 + a_3 \grs_3$, then $\grg \in \grG^\grS_{\grD_\gop(e)}$ if and only if 
$$\max\{a_2,a_3\} \leq a_1 \leq a_2+a_3.$$
Therefore $\grG^\grS_{\grD_\gop(e)}$ is generated by the elements
$$
 2D_1 - D_5 = \grs_1+\grs_2, \quad  D_1 -D_2 +D_4 = \grs_1+ \grs_3, \quad 2D_1 -D_3+D_4 = \grs_1 + \grs_2 + \grs_3.
$$

\subsection{Case 8.6} Notice that this case is a parabolic induction of the comodel wonderful variety of cotype $\sfE_8$, which was treated in \cite[Section 8]{BGM}.
We have in this case $\grS = \{\grs_1, \ldots, \grs_7\}$ and $\grD = \{D_1, \ldots, D_9\}$, where we label the spherical roots as $\grs_1 = \gra_4$, $\grs_2 = \gra_5$, $\grs_3 = \gra_8$, $\grs_4 = \gra_3$,  $\grs_5 = \gra_6$, $\grs_6 = \gra_2$, $\grs_7 = \gra_7$, and the colors as $D_1 = D_{\gra_4}^+$, $D_2 = D_{\gra_5}^-$, $D_3 = D_{\gra_4}^-$, $D_4 = D_{\gra_3}^+$, $D_5 = D_{\gra_3}^-$, $D_6 = D_{\gra_6}^+$, $D_7 = D_{\gra_7}^-$, $D_8 = D_{\gra_7}^+$, $D_9 = D_{\gra_1}$. Then we have the following relations between colors and spherical roots:
\begin{align*}
&	\grs_1 = D_1 +D_3 -D_4 \\
& \grs_2 = D_2 -D_3 +D_4 -D_5 \\
& \grs_3 = -D_1 -D_2 +D_3 +D_4 -D_5 \\
& \grs_4 = -D_2 -D_3 +D_4 +D_5 -D_6 \\
& \grs_5 = -D_4 +D_5 +D_6 -D_7 \\
& \grs_6 = -D_5 +D_6 +D_7 -D_8 -D_9 \\
& \grs_7 = -D_6 +D_7 +D_8
\end{align*}

Notice that $D_\gop = D_8$, which is a minuscule element in $\mN\grD$. Therefore $\ol{Ke}$ is normal. Let $D \in \grG_{D_8}$, and write $D =  \sum_{i=1}^9 c_i D_i$. Then we have the equality $c_9 = c_2+c_3+c_4+c_5$. On the other hand, notice that
\begin{align*}
&D_1 = 2D_8 - ( \sigma_2 + \sigma_3 + 2\sigma_5 + 2\sigma_7)\\
&D_7 = 3D_8 - ( 2\sigma_1 + 3\sigma_2 + 2\sigma_3 + \sigma_4 + 4\sigma_5 + 3\sigma_7)\\
&D_6 = 4D_8 - ( 2\sigma_1 + 3\sigma_2 + 2\sigma_3 + \sigma_4 + 4\sigma_5 + 4\sigma_7)\\
&D_2 +D_9 = 4D_8 - ( 3\sigma_1 + 4\sigma_2 + 3\sigma_3 + 2\sigma_4 + 6\sigma_5 + \sigma_6 + 5\sigma_7) \\
&D_3 +D_9 = 5D_8 - ( 3\sigma_1 + 5\sigma_2 + 3\sigma_3 + 2\sigma_4 + 7\sigma_5 + \sigma_6 + 6\sigma_7)\\
&D_5 +D_9 = 6D_8 - ( 4\sigma_1 + 6\sigma_2 + 4\sigma_3 + 2\sigma_4 + 8\sigma_5 + \sigma_6 + 7\sigma_7)\\
&D_4 +D_9 = 7D_8 - ( 4\sigma_1 + 6\sigma_2 + 4\sigma_3 + 2\sigma_4 + 9\sigma_5 + \sigma_6 + 8\sigma_7)
\end{align*}
Therefore $\grG_{D_1}$ is generated by the elements in the previous list, together with $D_8$.

\subsection{Case 12.2} In this case, we have $\grS = \{\gra\}$ and $\grD = \{D_1, D_2, D_3\}$, where we enumerate the colors as $D_1=D_\alpha^+$, $D_2=D_\alpha^-$, $D_3=D_{\alpha'}$. There is a unique relation relating colors and spherical roots, namely
\[\alpha = D_1+D_2.\]
Recall that $\gop=V(3\omega+\omega')$, thus we must have $\omega(D_\gop)=3\omega+\omega'$. It follows that, up to an equivariant automorphism of $X$, we have either $D_\gop=2D_1+D_2+D_3$, or $D_\gop = 3D_1+D_3$. On the other hand by construction the equivariant morphism $\phi : X \lra \mP(\gop)$ defined at the beginning of the section is birational, thus by \cite[\S 2.4.3]{BL} it follows that $D_\gop=2D_1+D_2+D_3$: indeed, $\{D_2\}$ is a distinguished subset, therefore $3D_1+D_3$ is not a faithful divisor and the corresponding morphism $X \lra \mP(V_{3D_1+D_3})$ is not birational.

Notice that $D_\gop$ is not minuscule, since $D_\gop-\gra =D_1+D_3$. Therefore, $\ol{Ke}$ is not normal.  The semigroup $\Gamma_{D_\gop}$ is generated by $2D_1+D_2+D_3$ and $D_1+D_3$.

%%%%%%%%%%%%%%%%%%%%%%%%%%%%%%%%%%%%%%%%%%%%%%%%%%%%%%%%%%%%%%%%%%%%%%%%%%%%%%%%%%%%%%%%%%%%%%%%%%%%%%%%%%%%%%%%%%%%%%%

\appendix

\section{List of spherical nilpotent $K$-orbits  in $\mathfrak p$\\in the exceptional cases}\label{A}

\renewcommand{\thesubsection}{\arabic{subsection}}

Here we enumerate the exceptional symmetric pairs $(\mathfrak g,\mathfrak k)$ from 1 to 12, the Cartan notations for the corresponding real forms of $\mathfrak g$ (type of $\mathfrak g$ followed by $\dim\mathfrak p-\dim\mathfrak k$) and the corresponding symmetric spaces (type of $\mathfrak g$ followed by a Roman numeral) are recalled in the tables of Appendix \ref{B}.   

We label the nilpotent $K$-orbits in $\mathfrak p$ by two numbers $n.m$, where $n$ denotes the symmetric pair and $m$ the number of the orbit in \makebox[0pt]{\rule{3pt}{0pt}\rule[4pt]{3pt}{0.8pt}}Dokovi\'c's tables \cite{D88a,D88b}. We also recall here the Kostant-Dynkin diagram of the orbit.

For every simple exceptional Lie algebra $\mathfrak g$ we fix a Chevalley basis 
\[\{h_\alpha\,:\,\alpha\in S\}\cup\{x_\alpha\,:\,\alpha\in R^+\}\cup\{y_\alpha\,:\,\alpha\in R^+\},\] 
where $S\subset R^+\subset R$ are respectively the set of the simple roots, the set of the positive roots and the set of the roots of $\mathfrak g$. In particular, $\{x_\alpha,h_\alpha,y_\alpha\}$ is an $\mathfrak{sl}(2)$ triple for all $\gra \in R^+$, and $h_\alpha = [x_\alpha,y_\alpha]$. We denote the elements of the irreducible exceptional root systems as in the tables of \cite{bourbaki}. 

For every exceptional symmetric pair $(\mathfrak g,\mathfrak k)$ we fix a Cartan subalgebra of $\mathfrak k$ included in the Cartan subalgebra of $\mathfrak g$, we fix a set of simple root vectors for $\mathfrak k$, and we give the highest weight vector(s) of $\mathfrak p$. When $\mathfrak k$ is of classical type, we also fix a basis of the standard representation to describe more explicitly $\mathfrak k$ and its representation on $\mathfrak p$.

For every $K$-orbit, we will provide a normal triple $\{e,h,f\}$. We computed these data with the help of SLA, a GAP package (see \cite{dG}), however in some cases we give here triples which are different from those obtained with SLA, since they are slightly more convenient for our computations. When $\mathfrak k$ is of classical type, the representative $e$ of the orbit is also described in classical terms.

One can then compute the centralizer $K_e$. Since it is enough for our purposes, to simplify the list we only give here its Lie algebra $\mathfrak k_e$.

Let $\mathfrak q=\mathfrak l +\mathfrak n$ be the parabolic subalgebra of $\mathfrak k$ given by
\[\mathfrak l=\mathfrak k(0),\quad \mathfrak n=\bigoplus_{i>0}\mathfrak k(i),\]
where $\mathfrak k(i)$ is the $\mathrm{ad}(h)$-eigenspace of eigenvalue $i$. It can easily be read off from the Kostant-Dynkin diagram of the orbit. We give $\mathfrak k_e$ as $\mathfrak l_e + \mathfrak m$, where $\mathfrak l_e$ the centralizer of $e$ in $\mathfrak l$, a Levi subalgebra, and $\mathfrak m$ is the nilpotent radical that we describe as a subalgebra of $\mathfrak n$.

\subsection{$\mathsf E_6/\mathsf C_4$}

We have $\mathfrak k=\mathfrak{sp}(8)$, $\mathfrak p=V(\omega_4)$.
Take $\theta$ such that $\alpha_2^\vee,\alpha_4^\vee,\alpha_3^\vee+\alpha_5^\vee,\alpha_1^\vee+\alpha_6^\vee$ form a basis of $\mathfrak t^\theta$ and take the following root vectors for the simple roots of $\mathfrak k$:
\[y_{\esix{1}{1}{1}{2}{1}{0}}-y_{\esix{0}{1}{1}{2}{1}{1}},\ 
-x_{\esix{1}{0}{1}{1}{0}{0}}+x_{\esix{0}{0}{0}{1}{1}{1}},\
y_{\esix{1}{0}{1}{1}{1}{0}}+y_{\esix{0}{0}{1}{1}{1}{1}},\
x_{\esix{1}{1}{2}{2}{2}{1}}.\]
Then $x_{\esix{1}{1}{1}{1}{1}{1}}$ is a highest weight vector in $\mathfrak p$.

Fix a basis $e_1,e_2,e_3,e_4,e_{-4},e_{-3},e_{-2},e_{-1}$ of $\mathbb C^8$, a skew-symmetric bilinear form $\omega$ such that $\omega(e_i,e_j)=\delta_{i,-j}$ and $\mathfrak k=\mathfrak{sp}(\mathbb C^8,\omega)$. We can identify $x_{\esix{1}{1}{1}{1}{1}{1}}=e_1\wedge e_2\wedge e_3\wedge e_4$ in $\mathfrak p\subset\mathsf\Lambda^4\mathbb C^8$. 

\subsubsection*{1.1. $(0001)$}

\[\left\{x_{\esix{1}{1}{1}{1}{1}{1}},\ h_{\esix{1}{1}{1}{1}{1}{1}},\ y_{\esix{1}{1}{1}{1}{1}{1}}\right\}\]
We can take $e=e_1\wedge e_2\wedge e_3\wedge e_4$. We have $\mathfrak l=\mathfrak k_h\cong\mathfrak{gl}(4)$ 
and 
$\mathfrak k_e=\mathfrak l_e+\mathfrak n$,
where
$\mathfrak l_e\cong\mathfrak{sl}(4)$.

\subsubsection*{1.2. $(0100)$}
 
\[\left\{x_{\esix{1}{0}{0}{0}{0}{0}}+x_{\esix{0}{0}{0}{0}{0}{1}},\ \alpha_1^\vee+\alpha_6^\vee,\ y_{\esix{1}{0}{0}{0}{0}{0}}+y_{\esix{0}{0}{0}{0}{0}{1}}\right\}\]
We can take $e=e_1\wedge e_2\wedge e_3\wedge e_{-3}-e_1\wedge e_2\wedge e_4\wedge e_{-4}$, and we have $\mathfrak l=\mathfrak k_h\cong\mathfrak{gl}(2)\oplus\mathfrak{sp}(4)$ 
and 
$\mathfrak k_e=\mathfrak l_e+\mathfrak n$. 
We have
$\mathfrak l_e\cong\mathfrak{sl}(2)\oplus\mathfrak{sl}(2)\oplus\mathfrak{sl}(2)$, one summand contained in $\mathfrak{gl}(2)$, the other ones in $\mathfrak{sp}(4)$.

\subsubsection*{1.3. $(1001)$}
 
\[\left\{x_{\esix{1}{0}{1}{1}{1}{1}}-y_{\esix{0}{0}{1}{1}{0}{0}}+y_{\esix{0}{0}{0}{1}{1}{0}},\ \alpha_1^\vee-\alpha_4^\vee+\alpha_6^\vee,\ y_{\esix{1}{0}{1}{1}{1}{1}}-x_{\esix{0}{0}{1}{1}{0}{0}}+x_{\esix{0}{0}{0}{1}{1}{0}}\right\}\]
We can take $e=e_1\wedge e_2\wedge e_3\wedge e_{-2}-e_1\wedge e_2\wedge e_4\wedge e_{-3}+e_1\wedge e_3\wedge e_4\wedge e_{-4}$. We have $\mathfrak l=\mathfrak k_h\cong\mathfrak{gl}(1)\oplus\mathfrak{gl}(3)$ 
and 
$\mathfrak k_e=\mathfrak l_e+\mathfrak m$, 
where
$\mathfrak l_e\cong\mathfrak{gl}(1)\oplus\mathfrak{so}(3)$, the $\mathfrak{so}(3)$ summand being contained in $\mathfrak{gl}(3)$, and $\mathfrak m$ is the $\mathfrak l_e$-complement in $\mathfrak n$ of the unique $1$-dimensional $\mathfrak l_e$-submodule of $\mathfrak k(1)$.

\subsection{$\mathsf E_6/\mathsf A_5 \times \mathsf A_1$}

We have $\mathfrak k=\mathfrak{sl}(6)\oplus\mathfrak{sl}(2)$, $\mathfrak p=V(\omega_3+\omega')$.
Take $\theta$ such that $\mathfrak t^\theta=\mathfrak t$ and take the following root vectors for the simple roots of $\mathfrak k$:
\[x_{\esix{1}{0}{0}{0}{0}{0}},\ x_{\esix{0}{0}{1}{0}{0}{0}},\ 
x_{\esix{0}{0}{0}{1}{0}{0}},\ x_{\esix{0}{0}{0}{0}{1}{0}},\
x_{\esix{0}{0}{0}{0}{0}{1}},\
x_{\esix{1}{2}{2}{3}{2}{1}}.\]
Then $x_{\esix{1}{1}{2}{3}{2}{1}}$ is a highest weight vector in $\mathfrak p$.

Fix a basis $e_1,\ldots,e_6$ of $V$, a basis $e'_1,e'_2$ of $W$, and $\mathfrak k=\mathfrak{sl}(V)\oplus\mathfrak{sl}(W)$. We can identify $x_{\esix{1}{1}{2}{3}{2}{1}}=e_1\wedge e_2\wedge e_3\otimes e'_1$ in $\mathfrak p=\left(\mathsf\Lambda^3V\right)\otimes W$.

\subsubsection*{2.1. $(00100;\ 1)$}
 
\[\left\{x_{\esix{1}{1}{2}{3}{2}{1}},\ h_{\esix{1}{1}{2}{3}{2}{1}},\ y_{\esix{1}{1}{2}{3}{2}{1}}\right\}\]
We can take $e=e_1\wedge e_2\wedge e_3\otimes e'_1$. We have $\mathfrak l=\mathfrak k_h\cong\mathfrak{s}(\mathfrak{gl}(3)\oplus\mathfrak{gl}(3))\oplus\mathfrak{gl}(1)$ 
and 
$\mathfrak k_e=\mathfrak l_e+\mathfrak n$,
where
$\mathfrak l_e\cong\mathfrak{gl}(1)\oplus\mathfrak{sl}(3)\oplus\mathfrak{sl}(3)$.

\subsubsection*{2.2. $(10001;\ 2)$}
 
\[\left\{x_{\esix{1}{1}{2}{2}{1}{1}}+x_{\esix{1}{1}{1}{2}{2}{1}},\ h_{\esix{1}{1}{2}{2}{1}{1}}+h_{\esix{1}{1}{1}{2}{2}{1}},\ y_{\esix{1}{1}{2}{2}{1}{1}}+y_{\esix{1}{1}{1}{2}{2}{1}}\right\}\]
We can take $e=(e_1\wedge e_2\wedge e_5+e_1\wedge e_3\wedge e_4)\otimes e'_1$. We have $\mathfrak l=\mathfrak k_h\cong\mathfrak{s}(\mathfrak{gl}(1)\oplus\mathfrak{gl}(4)\oplus\mathfrak{gl}(1))\oplus\mathfrak{gl}(1)$ 
and 
$\mathfrak k_e=\mathfrak l_e+\mathfrak n$,
where
$\mathfrak l_e\cong\mathfrak{gl}(1)^{\oplus2}\oplus\mathfrak{sp}(4)$.

\subsubsection*{2.3. $(01010;\ 0)$}
 
\[\left\{x_{\esix{1}{1}{2}{2}{2}{1}}+y_{\esix{0}{1}{0}{0}{0}{0}},\ h_{\esix{1}{1}{2}{2}{2}{1}}-h_{\esix{0}{1}{0}{0}{0}{0}},\ y_{\esix{1}{1}{2}{2}{2}{1}}+x_{\esix{0}{1}{0}{0}{0}{0}}\right\}\]
We can take $e=e_1\wedge e_2\wedge e_3\otimes e'_2-e_1\wedge e_2\wedge e_4\otimes e'_1$. We have $\mathfrak l=\mathfrak k_h\cong\mathfrak{s}(\mathfrak{gl}(2)\oplus\mathfrak{gl}(2)\oplus\mathfrak{gl}(2))\oplus\mathfrak{sl}(2)$ 
and 
$\mathfrak k_e=\mathfrak l_e+\mathfrak n$,
where
$\mathfrak l_e\cong\mathfrak{gl}(1)\oplus\mathfrak{sl}(2)\oplus\mathfrak{sl}(2)\oplus\mathfrak{sl}(2)$, with $\mathfrak{sl}(2)^{\oplus3}$ embedded diagonally into $\mathfrak{sl}(2)^{\oplus4}$ via $(A,B,C)\longmapsto (A,B,C,B)$.

\subsubsection*{2.4. $(00100;\ 3)$}
 
\[\begin{array}{rcl}
\bigg\{ & x_{\esix{1}{1}{2}{2}{1}{0}}+x_{\esix{1}{1}{1}{2}{1}{1}}+x_{\esix{0}{1}{1}{2}{2}{1}},\qquad h_{\esix{1}{1}{2}{2}{1}{0}}+h_{\esix{1}{1}{1}{2}{1}{1}}+h_{\esix{0}{1}{1}{2}{2}{1}},\\
 & y_{\esix{1}{1}{2}{2}{1}{0}}+y_{\esix{1}{1}{1}{2}{1}{1}}+y_{\esix{0}{1}{1}{2}{2}{1}} & \bigg\}
\end{array}\]
We can take $e=(e_1\wedge e_2\wedge e_6-e_1\wedge e_3\wedge e_5+e_2\wedge e_3\wedge e_4)\otimes e'_1$. We have $\mathfrak l=\mathfrak k_h\cong\mathfrak{s}(\mathfrak{gl}(3)\oplus\mathfrak{gl}(3))\oplus\mathfrak{gl}(1)$ 
and 
$\mathfrak k_e=\mathfrak l_e+\mathfrak m$, 
where
$\mathfrak l_e\cong\mathfrak{gl}(1)\oplus\mathfrak{sl}(3)$, the $\mathfrak{sl}(3)$ summand embedded diagonally in $\mathfrak{sl}(3)^{\oplus2}$, and $\mathfrak m$ is the $\mathfrak l_e$-complement in $\mathfrak n$ of the unique $1$-dimensional $\mathfrak l_e$-submodule in $\mathfrak k(1)\cap\mathfrak{sl}(V)$.

\subsubsection*{2.5. $(10101;\ 1)$}
 
\[\begin{array}{rcl}
\bigg\{ & x_{\esix{1}{1}{2}{2}{1}{1}}+x_{\esix{1}{1}{1}{2}{2}{1}}+y_{\esix{0}{1}{0}{0}{0}{0}},\qquad h_{\esix{1}{1}{2}{2}{1}{1}}+h_{\esix{1}{1}{1}{2}{2}{1}}-h_{\esix{0}{1}{0}{0}{0}{0}},\\ 
 & y_{\esix{1}{1}{2}{2}{1}{1}}+y_{\esix{1}{1}{1}{2}{2}{1}}+x_{\esix{0}{1}{0}{0}{0}{0}} & \bigg\}
\end{array}\]
We can take $e=e_1\wedge e_2\wedge e_3\otimes e'_2-(e_1\wedge e_2\wedge e_5-e_1\wedge e_3\wedge e_4)\otimes e'_1$. We have $\mathfrak l=\mathfrak k_h\cong\mathfrak{s}(\mathfrak{gl}(1)\oplus\mathfrak{gl}(2)\oplus\mathfrak{gl}(2)\oplus\mathfrak{gl}(1))\oplus\mathfrak{gl}(1)$ 
and 
$\mathfrak k_e=\mathfrak l_e+\mathfrak m$, 
where
$\mathfrak l_e\cong\mathfrak{gl}(1)^{\oplus2}\oplus\mathfrak{sl}(2)$, the $\mathfrak{sl}(2)$ summand embedded diagonally in $\mathfrak{sl}(2)^{\oplus2}$. The $\mathfrak l_e$-submodule $\mathfrak m\subset\mathfrak n$ can be described as follows. In $\mathfrak k(1)$ there are precisely two $1$-dimensional $\mathfrak l_e$-submodules lying into different isotypic $\mathfrak l$-components, call them $\mathfrak u_1$ and $\mathfrak u_2$: then $\mathfrak m$ is the sum of the $\mathfrak l_e$-complement of $\mathfrak u_1\oplus\mathfrak u_2$ in $\mathfrak n$, plus a one dimensional subspace $\mathfrak u\subset\mathfrak u_1\oplus\mathfrak u_2$ that projects nontrivially on both summands (which are isomorphic as $\mathfrak l_e$-modules).

\subsection{$\mathsf E_6/\mathsf D_5\times \mathbb C^\times$}

We have $\mathfrak k=\mathfrak{so}(10)\oplus\mathfrak{gl}(1)$, $\mathfrak p=V(\omega_4)\oplus V(\omega_5)$ as $\mathfrak k^{\mathrm{ss}}$-module.
Take $\theta$ such that $\mathfrak t^\theta=\mathfrak t$ and take the following root vectors for the simple roots of $\mathfrak k$:
\[x_{\esix{0}{0}{0}{0}{0}{1}},\ x_{\esix{0}{0}{0}{0}{1}{0}},\ 
x_{\esix{0}{0}{0}{1}{0}{0}},\ x_{\esix{0}{0}{1}{0}{0}{0}},\
x_{\esix{0}{1}{0}{0}{0}{0}}.\]
Then $y_{\esix{1}{0}{0}{0}{0}{0}}$ and $x_{\esix{1}{2}{2}{3}{2}{1}}$ are highest weight vectors in $\mathfrak p$.

Fix a basis $e_1,\ldots,e_5$ of $W$ and a dual basis $\varphi_1,\ldots,\varphi_5$ of $W^*$, so that we have a nondegenerate symmetric bilinear form on $V=W\oplus W^*$, and $\mathfrak k=\mathfrak{so}(V)$. The spin representations can be realized in $\mathsf\Lambda W^*$ (as the odd and even degree parts) via the anti-symmetric square 
\[\sigma^2\colon \mathsf\Lambda^2V\otimes\mathsf\Lambda W^*\to\mathsf\Lambda W^*\] 
of the contraction-extension map
\[\sigma\colon V\otimes\mathsf\Lambda W^*\to\mathsf\Lambda W^*\]
and the natural isomorphism between $\mathfrak{so}(V)$ and $\mathsf\Lambda^2V$. We can identify $y_{\esix{1}{0}{0}{0}{0}{0}}=\varphi_5$ and $x_{\esix{1}{2}{2}{3}{2}{1}}=1$ in $\mathfrak p=\mathsf\Lambda W^*$. 

\subsubsection*{3.1. $(00001;\ 0)$}

\[\left\{x_{\esix{1}{2}{2}{3}{2}{1}},\ h_{\esix{1}{2}{2}{3}{2}{1}},\ y_{\esix{1}{2}{2}{3}{2}{1}}\right\}\]
We can take $e=1$. We have $\mathfrak l=\mathfrak k_h\cong\mathfrak{gl}(1)\oplus\mathfrak{gl}(5)$ 
and 
$\mathfrak k_e=\mathfrak l_e+\mathfrak n$, 
where
$\mathfrak l_e\cong\mathfrak{gl}(1)\oplus\mathfrak{sl}(5)$.

\subsubsection*{3.2. $(00010;\ -2)$}

\[\left\{y_{\esix{1}{0}{0}{0}{0}{0}},\ -h_{\esix{1}{0}{0}{0}{0}{0}},\ x_{\esix{1}{0}{0}{0}{0}{0}}\right\}\]
We can take $e=\varphi_5$. This case is equal to the previous one up to an external automorphism of $\mathfrak k$.

\subsubsection*{3.3. $(10000;\ 1)$}

\[\left\{x_{\esix{1}{1}{2}{2}{1}{1}}+x_{\esix{1}{1}{1}{2}{2}{1}},\ h_{\esix{1}{1}{2}{2}{1}{1}}+h_{\esix{1}{1}{1}{2}{2}{1}},\ y_{\esix{1}{1}{2}{2}{1}{1}}+y_{\esix{1}{1}{1}{2}{2}{1}}\right\}\]
We can take $e=\varphi_4\wedge\varphi_3 +\varphi_5\wedge\varphi_2$, and we have $\mathfrak l=\mathfrak k_h\cong\mathfrak{gl}(1)^{\oplus2}\oplus\mathfrak{so}(8)$ 
and 
$\mathfrak k_e=\mathfrak l_e+\mathfrak n$, 
where
$\mathfrak l_e\cong\mathfrak{gl}(1)\oplus\mathfrak{so}(7)$.

\subsubsection*{3.4. $(10000;\ -2)$}

\[\left\{y_{\esix{1}{1}{1}{1}{0}{0}}+y_{\esix{1}{0}{1}{1}{1}{0}},\ -h_{\esix{1}{1}{1}{1}{0}{0}}-h_{\esix{1}{0}{1}{1}{1}{0}},\ x_{\esix{1}{1}{1}{1}{0}{0}}+x_{\esix{1}{0}{1}{1}{1}{0}}\right\}\]
We can take $e=\varphi_5\wedge\varphi_4\wedge\varphi_3 -\varphi_2$. This case is equal to the previous one up to an external automorphism of $\mathfrak k$.

\subsubsection*{3.5. $(00011;\ -2)$}

\[\left\{x_{\esix{1}{2}{2}{3}{2}{1}}+y_{\esix{1}{0}{0}{0}{0}{0}},\ h_{\esix{1}{2}{2}{3}{2}{1}}-h_{\esix{1}{0}{0}{0}{0}{0}},\ y_{\esix{1}{2}{2}{3}{2}{1}}+x_{\esix{1}{0}{0}{0}{0}{0}}\right\}\]
We can take $e=\varphi_5+1$. We have $\mathfrak l=\mathfrak k_h\cong\mathfrak{gl}(1)\oplus\mathfrak{gl}(4)\oplus\mathfrak{gl}(1)$ 
and 
$\mathfrak k_e=\mathfrak l_e+\mathfrak n$, 
where
$\mathfrak l_e\cong\mathfrak{gl}(1)\oplus\mathfrak{sl}(4)$.

\subsubsection*{3.6. $(02000;\ -2)$}

\[\left\{x_{\esix{1}{2}{2}{3}{2}{1}}+y_{\esix{1}{1}{1}{1}{0}{0}},\ 2h_{\esix{1}{2}{2}{3}{2}{1}}-2h_{\esix{1}{1}{1}{1}{0}{0}},\ 2y_{\esix{1}{2}{2}{3}{2}{1}}+2x_{\esix{1}{1}{1}{1}{0}{0}}\right\}\]
We can take $e=\varphi_5\wedge\varphi_4\wedge\varphi_3+1$. We have $\mathfrak l=\mathfrak k_h\cong\mathfrak{gl}(1)\oplus\mathfrak{gl}(2)\oplus\mathfrak{so}(6)$ 
and 
$\mathfrak k_e=\mathfrak l_e+\mathfrak n$, 
where
$\mathfrak l_e\cong\mathfrak{gl}(1)\oplus\mathfrak{sl}(2)\oplus\mathfrak{sl}(3)$.

\subsubsection*{3.7. $(11010;\ -2)$}

\[\begin{array}{rcl}
\bigg\{ & x_{\esix{1}{1}{2}{2}{1}{1}}+x_{\esix{1}{1}{1}{2}{2}{1}}+y_{\esix{1}{0}{0}{0}{0}{0}},\qquad h_{\esix{1}{1}{2}{2}{1}{1}}+2h_{\esix{1}{1}{1}{2}{2}{1}}-2h_{\esix{1}{0}{0}{0}{0}{0}},\\ 
 & y_{\esix{1}{1}{2}{2}{1}{1}}+2y_{\esix{1}{1}{1}{2}{2}{1}}+2x_{\esix{1}{0}{0}{0}{0}{0}} & \big\}
\end{array}\]
We can take $e=\varphi_4\wedge\varphi_3 +\varphi_5\wedge\varphi_2+\varphi_5$. We have $\mathfrak l=\mathfrak k_h\cong\mathfrak{gl}(1)^{\oplus3}\oplus\mathfrak{gl}(3)$ 
and 
$\mathfrak k_e=\mathfrak l_e+\mathfrak m$, 
where
$\mathfrak l_e\cong\mathfrak{gl}(1)\oplus\mathfrak{sl}(3)$. The $\mathfrak l_e$-submodule $\mathfrak m\subset\mathfrak n$ can be described as follows. In $\mathfrak k(1)$ there are precisely two $3$-dimensional simple $\mathfrak l$-submodules, call them $\mathfrak u_1$ and $\mathfrak u_2$, and they are isomorphic as $\mathfrak l_e$-modules: then $\mathfrak m$ is the sum of the $\mathfrak l_e$-complement of $\mathfrak u_1\oplus\mathfrak u_2$ in $\mathfrak n$, plus a simple $\mathfrak l_e$-submodule $\mathfrak u\subset\mathfrak u_1\oplus\mathfrak u_2$ which projects nontrivially on both summands.

\subsubsection*{3.8. $(11001;\ -3)$}

\[\begin{array}{rcl}
\bigg\{ & x_{\esix{1}{2}{2}{3}{2}{1}}+y_{\esix{1}{1}{1}{1}{0}{0}}+y_{\esix{1}{0}{1}{1}{1}{0}},\qquad 2h_{\esix{1}{2}{2}{3}{2}{1}}-2h_{\esix{1}{1}{1}{1}{0}{0}}-h_{\esix{1}{0}{1}{1}{1}{0}},\\ 
 & 2y_{\esix{1}{2}{2}{3}{2}{1}}+2x_{\esix{1}{1}{1}{1}{0}{0}}+x_{\esix{1}{0}{1}{1}{1}{0}} & \bigg\}
\end{array}\]
We can take $e=\varphi_5\wedge\varphi_4\wedge\varphi_3 -\varphi_2+1$. This case is equal to the previous one up to an external automorphism of $\mathfrak k$.

\subsubsection*{3.9. $(40000;\ -2)$}

\[\begin{array}{rcl}
\bigg\{ & x_{\esix{1}{1}{2}{2}{1}{1}}+x_{\esix{1}{1}{1}{2}{2}{1}}+y_{\esix{1}{1}{1}{1}{0}{0}}+y_{\esix{1}{0}{1}{1}{1}{0}},\\ 
\rule{0pt}{15pt} & 2h_{\esix{1}{1}{2}{2}{1}{1}}+2h_{\esix{1}{1}{1}{2}{2}{1}}-2h_{\esix{1}{1}{1}{1}{0}{0}}-2h_{\esix{1}{0}{1}{1}{1}{0}},\\
& 2y_{\esix{1}{1}{2}{2}{1}{1}}+2y_{\esix{1}{1}{1}{2}{2}{1}}+2x_{\esix{1}{1}{1}{1}{0}{0}}+2x_{\esix{1}{0}{1}{1}{1}{0}} & \bigg\}\end{array}\]
We can take $e=\varphi_5\wedge\varphi_4\wedge\varphi_3+\varphi_4\wedge\varphi_3 +\varphi_5\wedge\varphi_2-\varphi_2$. We have $\mathfrak l=\mathfrak k_h\cong\mathfrak{gl}(1)^{\oplus2}\oplus\mathfrak{so}(8)$ 
and 
$\mathfrak k_e=\mathfrak l_e+\mathfrak n$, 
where
$\mathfrak l_e\cong\mathfrak g_2$.

\subsection{$\mathsf E_6/\mathsf F_4$}

We have $\mathfrak k=\mathfrak f_4$, $\mathfrak p=V(\omega_4)$.
Take $\theta$ such that $\alpha_2^\vee,\alpha_4^\vee,\alpha_3^\vee+\alpha_5^\vee,\alpha_1^\vee+\alpha_6^\vee$ form a basis of $\mathfrak t^\theta$ and take the following root vectors for the simple roots of $\mathfrak k$:
\[y_{\esix{0}{0}{0}{1}{0}{0}},\ 
y_{\esix{1}{1}{2}{2}{2}{1}},\
x_{\esix{1}{1}{2}{2}{1}{1}}-x_{\esix{1}{1}{1}{2}{2}{1}},\
y_{\esix{1}{0}{0}{0}{0}{0}}+y_{\esix{0}{0}{0}{0}{0}{1}}.\]
Then $x_{\esix{0}{1}{1}{1}{0}{0}}+x_{\esix{0}{1}{0}{1}{1}{0}}$ is a highest weight vector in $\mathfrak p$.

\subsubsection*{4.1. $(0001)$}
 
\[\left\{x_{\esix{0}{1}{1}{1}{0}{0}}+x_{\esix{0}{1}{0}{1}{1}{0}},\ h_{\esix{0}{1}{1}{1}{0}{0}}+h_{\esix{0}{1}{0}{1}{1}{0}},\ y_{\esix{0}{1}{1}{1}{0}{0}}+y_{\esix{0}{1}{0}{1}{1}{0}}\right\}\]
We have $\mathfrak l=\mathfrak k_h\cong\mathfrak{gl}(1)\oplus\mathfrak{so}(7)$ 
and 
$\mathfrak k_e=\mathfrak l_e+\mathfrak n$, 
where
$\mathfrak l_e\cong\mathfrak{so}(7)$.

\subsection{$\mathsf E_7/\mathsf A_7$}

We have $\mathfrak k=\mathfrak{sl}(8)$, $\mathfrak p=V(\omega_4)$.
Take $\theta$ such that $\mathfrak t^\theta=\mathfrak t$ and take the following root vectors for the simple roots of $\mathfrak k$:
\[x_{\eseven{1}{0}{0}{0}{0}{0}{0}},\ x_{\eseven{0}{0}{1}{0}{0}{0}{0}},\ 
x_{\eseven{0}{0}{0}{1}{0}{0}{0}},\ x_{\eseven{0}{0}{0}{0}{1}{0}{0}},\
x_{\eseven{0}{0}{0}{0}{0}{1}{0}},\ x_{\eseven{0}{0}{0}{0}{0}{0}{1}},\
x_{\eseven{1}{2}{2}{3}{2}{1}{0}}.\]
Then $x_{\eseven{1}{1}{2}{3}{3}{2}{1}}$ is a highest weight vector in $\mathfrak p$.

Fix a basis $e_1,\ldots,e_8$ of $\mathbb C^8$, and $\mathfrak k=\mathfrak{sl}(\mathbb C^8)$. We can identify $x_{\eseven{1}{1}{2}{3}{3}{2}{1}}=e_1\wedge e_2\wedge e_3\wedge e_4$ in $\mathfrak p=\mathsf\Lambda^4\mathbb C^8$.

\subsubsection*{5.1. $(0001000)$}
 
\[\left\{x_{\eseven{1}{1}{2}{3}{3}{2}{1}},\ h_{\eseven{1}{1}{2}{3}{3}{2}{1}},\ y_{\eseven{1}{1}{2}{3}{3}{2}{1}}\right\}\]
Can take $e=e_1\wedge e_2\wedge e_3\wedge e_4$, we have $\mathfrak l=\mathfrak k_h\cong\mathfrak{s}(\mathfrak{gl}(4)\oplus\mathfrak{gl}(4))$ 
and 
$\mathfrak k_e=\mathfrak l_e+\mathfrak n$ 
where
$\mathfrak l_e\cong\mathfrak{sl}(4)\oplus\mathfrak{sl}(4)$.

\subsubsection*{5.2. $(0100010)$}
 
\[\left\{x_{\eseven{1}{1}{2}{3}{2}{1}{1}}+x_{\eseven{1}{1}{2}{2}{2}{2}{1}},\ h_{\eseven{1}{1}{2}{3}{2}{1}{1}}+h_{\eseven{1}{1}{2}{2}{2}{2}{1}},\ y_{\eseven{1}{1}{2}{3}{2}{1}{1}}+y_{\eseven{1}{1}{2}{2}{2}{2}{1}}\right\}\]
We can take $e=e_1\wedge e_2\wedge e_3\wedge e_6 + e_1\wedge e_2\wedge e_4\wedge e_5$. We have $\mathfrak l=\mathfrak k_h\cong\mathfrak{s}(\mathfrak{gl}(2)\oplus\mathfrak{gl}(4)\oplus\mathfrak{gl}(2))$ 
and 
$\mathfrak k_e=\mathfrak l_e+\mathfrak n$, 
where
$\mathfrak l_e\cong\mathfrak{gl}(1)\oplus\mathfrak{sl}(2)\oplus\mathfrak{sp}(4)\oplus\mathfrak{sl}(2)$.

\subsubsection*{5.3. $(0200000)$}
 
\[\begin{array}{rcl}
\bigg\{ & x_{\eseven{1}{1}{2}{2}{2}{1}{0}}+x_{\eseven{1}{1}{2}{2}{1}{1}{1}}+y_{\eseven{0}{1}{0}{0}{0}{0}{0}},\qquad h_{\eseven{1}{1}{2}{2}{2}{1}{0}}+h_{\eseven{1}{1}{2}{2}{1}{1}{1}}-h_{\eseven{0}{1}{0}{0}{0}{0}{0}},\\  & y_{\eseven{1}{1}{2}{2}{2}{1}{0}}+y_{\eseven{1}{1}{2}{2}{1}{1}{1}}+x_{\eseven{0}{1}{0}{0}{0}{0}{0}} & \bigg\}
\end{array}\]
We can take $e=e_1\wedge e_2\wedge e_3\wedge e_8 + e_1\wedge e_2\wedge e_4\wedge e_7 + e_1\wedge e_2\wedge e_5\wedge e_6$. We have $\mathfrak l=\mathfrak k_h\cong\mathfrak{s}(\mathfrak{gl}(2)\oplus\mathfrak{gl}(6))$ 
and 
$\mathfrak k_e=\mathfrak l_e+\mathfrak n$, 
where
$\mathfrak l_e\cong\mathfrak{sl}(2)\oplus\mathfrak{sp}(6)$.

\subsubsection*{5.4. $(0000020)$}
 
\[\begin{array}{rcl}
\bigg\{ & x_{\eseven{1}{1}{2}{2}{1}{1}{1}}+x_{\eseven{1}{1}{1}{2}{2}{1}{1}}+x_{\eseven{0}{1}{1}{2}{2}{2}{1}},\qquad h_{\eseven{1}{1}{2}{2}{1}{1}{1}}+h_{\eseven{1}{1}{1}{2}{2}{1}{1}}+h_{\eseven{0}{1}{1}{2}{2}{2}{1}},\\  & y_{\eseven{1}{1}{2}{2}{1}{1}{1}}+y_{\eseven{1}{1}{1}{2}{2}{1}{1}}+y_{\eseven{0}{1}{1}{2}{2}{2}{1}} & \bigg\}
\end{array}\]
We can take $e=e_1\wedge e_2\wedge e_5\wedge e_6 + e_1\wedge e_3\wedge e_4\wedge e_6 + e_2\wedge e_3\wedge e_4\wedge e_5$. This case is equal to the previous one up to an external automorphism of $\mathfrak k$.

\subsubsection*{5.5. $(1001001)$}
 
\[\begin{array}{rcl}
\bigg\{ & x_{\eseven{1}{1}{2}{3}{2}{1}{0}}+x_{\eseven{1}{1}{2}{2}{2}{1}{1}}+x_{\eseven{1}{1}{1}{2}{2}{2}{1}},\qquad h_{\eseven{1}{1}{2}{3}{2}{1}{0}}+h_{\eseven{1}{1}{2}{2}{2}{1}{1}}+h_{\eseven{1}{1}{1}{2}{2}{2}{1}},\\  & y_{\eseven{1}{1}{2}{3}{2}{1}{0}}+y_{\eseven{1}{1}{2}{2}{2}{1}{1}}+y_{\eseven{1}{1}{1}{2}{2}{2}{1}} & \bigg\}
\end{array}\]
We can take $e=e_1\wedge e_2\wedge e_3\wedge e_7 - e_1\wedge e_2\wedge e_4\wedge e_6 + e_1\wedge e_3\wedge e_4\wedge e_5$. We have $\mathfrak l=\mathfrak k_h\cong\mathfrak s(\mathfrak{gl}(1)\oplus\mathfrak{gl}(3)\oplus\mathfrak{gl}(3)\oplus\mathfrak{gl}(1))$ 
and 
$\mathfrak k_e=\mathfrak l_e+\mathfrak m$, 
where
$\mathfrak l_e\cong\mathfrak{gl}(1)^{\oplus2}\oplus\mathfrak{sl}(3)$, the $\mathfrak{sl}(3)$ summand embedded diagonally in $\mathfrak{sl}(3)^{\oplus2}$, and $\mathfrak m$ is the $\mathfrak l_e$-complement in $\mathfrak n$ of the unique $1$-dimensional $\mathfrak l_e$-submodule of $\mathfrak k(1)$.

\subsubsection*{5.8. $(1100100)$}
 
\[\begin{array}{rcl}
\bigg\{ & x_{\eseven{1}{1}{2}{2}{2}{1}{0}}+x_{\eseven{1}{1}{2}{2}{1}{1}{1}}+x_{\eseven{1}{1}{1}{2}{2}{2}{1}}+y_{\eseven{0}{1}{0}{0}{0}{0}{0}},\\ 
\rule{0pt}{15pt} & h_{\eseven{1}{1}{2}{2}{2}{1}{0}}+h_{\eseven{1}{1}{2}{2}{1}{1}{1}}+h_{\eseven{1}{1}{1}{2}{2}{2}{1}}-h_{\eseven{0}{1}{0}{0}{0}{0}{0}},\\
& y_{\eseven{1}{1}{2}{2}{2}{1}{0}}+y_{\eseven{1}{1}{2}{2}{1}{1}{1}}+y_{\eseven{1}{1}{1}{2}{2}{2}{1}}+x_{\eseven{0}{1}{0}{0}{0}{0}{0}} & \bigg\}
\end{array}\]
We can take $e=e_1\wedge e_2\wedge e_3\wedge e_8 + e_1\wedge e_2\wedge e_4\wedge e_7 + e_1\wedge e_2\wedge e_5\wedge e_6 + e_1\wedge e_3\wedge e_4\wedge e_5$.
We have $\mathfrak l=\mathfrak k_h\cong\mathfrak{s}(\mathfrak{gl}(1)\oplus\mathfrak{gl}(1)\oplus\mathfrak{gl}(3)\oplus\mathfrak{gl}(3))$ 
and 
$\mathfrak k_e=\mathfrak l_e+\mathfrak m$, 
where
$\mathfrak l_e\cong\mathfrak{gl}(1)\oplus\mathfrak{sl}(3)$, the $\mathfrak{sl}(3)$ summand embedded skew-diagonally in $\mathfrak{sl}(3)^{\oplus2}$ via $A\mapsto(A,-A^{\mathrm s})$. The $\mathfrak l_e$-submodule $\mathfrak m\subset\mathfrak n$ can be described as follows. In $\mathfrak k(1)$ there are precisely two $\mathfrak l_e$-submodules of dimension 3 lying into different isotypic $\mathfrak l$-components, call them $\mathfrak u_1$ and $\mathfrak u_2$, and they are simple and isomorphic as $\mathfrak l_e$-modules: then $\mathfrak m$ is the sum of the $\mathfrak l_e$-complement of $\mathfrak u_1\oplus\mathfrak u_2$ in $\mathfrak n$, plus a simple  $\mathfrak l_e$-submodule $\mathfrak u\subset\mathfrak u_1\oplus\mathfrak u_2$ that projects nontrivially on both summands.

\subsubsection*{5.9. $(0010011)$}
 
\[\begin{array}{rll}
\bigg\{ & x_{\eseven{1}{1}{2}{3}{2}{1}{0}}+x_{\eseven{1}{1}{2}{2}{1}{1}{1}}+x_{\eseven{1}{1}{1}{2}{2}{1}{1}}+x_{\eseven{0}{1}{1}{2}{2}{2}{1}},\\ 
\rule{0pt}{15pt} & h_{\eseven{1}{1}{2}{3}{2}{1}{0}}+h_{\eseven{1}{1}{2}{2}{1}{1}{1}}+h_{\eseven{1}{1}{1}{2}{2}{1}{1}}+h_{\eseven{0}{1}{1}{2}{2}{2}{1}},\\
& y_{\eseven{1}{1}{2}{3}{2}{1}{0}}+y_{\eseven{1}{1}{2}{2}{1}{1}{1}}+y_{\eseven{1}{1}{1}{2}{2}{1}{1}}+y_{\eseven{0}{1}{1}{2}{2}{2}{1}} & \bigg\}
\end{array}\]
We can take $e=e_1\wedge e_2\wedge e_3\wedge e_7 + e_1\wedge e_2\wedge e_5\wedge e_6 + e_1\wedge e_3\wedge e_4\wedge e_6 + e_2\wedge e_3\wedge e_4\wedge e_5$. This case is equal to the previous one up to an external automorphism of $\mathfrak k$.

\subsection{$\mathsf E_7/\mathsf D_6 \times \mathsf A_1$}

We have $\mathfrak k=\mathfrak{so}(12)+\mathfrak{sl}(2)$, $\mathfrak p=V(\omega_5+\omega')$.
Take $\theta$ such that $\mathfrak t^\theta=\mathfrak t$ and take the following root vectors for the simple roots of $\mathfrak k$:
\[x_{\eseven{0}{0}{0}{0}{0}{0}{1}},\ x_{\eseven{0}{0}{0}{0}{0}{1}{0}},\ 
x_{\eseven{0}{0}{0}{0}{1}{0}{0}},\ x_{\eseven{0}{0}{0}{1}{0}{0}{0}},\
x_{\eseven{0}{0}{1}{0}{0}{0}{0}},\ x_{\eseven{0}{1}{0}{0}{0}{0}{0}},\ 
x_{\eseven{2}{2}{3}{4}{3}{2}{1}}.\]
Then $x_{\eseven{1}{2}{3}{4}{3}{2}{1}}$ is a highest weight vector in $\mathfrak p$.

Fix a basis $e_1,\ldots,e_6$ of $U$ and a dual basis $\varphi_1,\ldots,\varphi_6$ of $U^*$, so that we have a nondegenerate symmetric bilinear form on $V=U\oplus U^*$. Fix also a basis $e'_1,e'_2$ of $W$. We have $\mathfrak k=\mathfrak{so}(V)\oplus\mathfrak{sl}(W)$. The spin representations can be realized in $\mathsf\Lambda U^*$ (as the odd and even degree parts) via the anti-symmetric square 
\[\sigma^2\colon \mathsf\Lambda^2V\otimes\mathsf\Lambda U^*\to\mathsf\Lambda U^*\] 
of the contraction-extension map
\[\sigma\colon V\otimes\mathsf\Lambda U^*\to\mathsf\Lambda U^*\]
and the natural isomorphism between $\mathfrak{so}(V)$ and $\mathsf\Lambda^2V$. We can identify $x_{\eseven{1}{2}{3}{4}{3}{2}{1}}=\varphi_6\otimes e'_1$ in $\mathfrak p=\left(\mathsf\Lambda^{\mathrm{odd}} U^*\right)\otimes W$. 

\subsubsection*{6.1. $(000010;\ 1)$}
 
\[\left\{x_{\eseven{1}{2}{3}{4}{3}{2}{1}},\ h_{\eseven{1}{2}{3}{4}{3}{2}{1}},\ y_{\eseven{1}{2}{3}{4}{3}{2}{1}}\right\}\]
We can take $e=\varphi_6\otimes e'_1$. We have $\mathfrak l=\mathfrak k_h\cong\mathfrak{gl}(6)\oplus\mathfrak{gl}(1)$ 
and 
$\mathfrak k_e=\mathfrak l_e+\mathfrak n$, 
where
$\mathfrak l_e\cong\mathfrak{gl}(1)\oplus\mathfrak{sl}(6)$.

\subsubsection*{6.2. $(010000;\ 2)$}
 
\[\left\{x_{\eseven{1}{2}{2}{3}{2}{2}{1}}+x_{\eseven{1}{1}{2}{3}{3}{2}{1}},\ h_{\eseven{1}{2}{2}{3}{2}{2}{1}}+h_{\eseven{1}{1}{2}{3}{3}{2}{1}},\ y_{\eseven{1}{2}{2}{3}{2}{2}{1}}+y_{\eseven{1}{1}{2}{3}{3}{2}{1}}\right\}\]
We can take $e=(\varphi_6\wedge\varphi_5\wedge\varphi_4 -\varphi_3)\otimes e'_1$. We have $\mathfrak l=\mathfrak k_h\cong\mathfrak{gl}(2)\oplus\mathfrak{so}(8)\oplus\mathfrak{gl}(1)$ 
and 
$\mathfrak k_e=\mathfrak l_e+\mathfrak n$, 
where
$\mathfrak l_e\cong\mathfrak{gl}(1)\oplus\mathfrak{sl}(2)\oplus\mathfrak{so}(7)$.

\subsubsection*{6.3. $(000100;\ 0)$}
 
\[\left\{x_{\eseven{1}{2}{2}{4}{3}{2}{1}}+y_{\eseven{1}{0}{0}{0}{0}{0}{0}},\ h_{\eseven{1}{2}{2}{4}{3}{2}{1}}-h_{\eseven{1}{0}{0}{0}{0}{0}{0}},\ y_{\eseven{1}{2}{2}{4}{3}{2}{1}}+x_{\eseven{1}{0}{0}{0}{0}{0}{0}}\right\}\]
We can take $e=\varphi_5\otimes e'_1 +\varphi_6\otimes e'_2$. We have $\mathfrak l=\mathfrak k_h\cong\mathfrak{gl}(4)\oplus\mathfrak{so}(4)\oplus\mathfrak{sl}(2)$ 
and 
$\mathfrak k_e=\mathfrak l_e+\mathfrak n$, 
where
$\mathfrak l_e\cong\mathfrak{sl}(4)\oplus\mathfrak{sl}(2)\oplus\mathfrak{sl}(2)$, $\mathfrak{sl}^{\oplus2}$ embedded diagonally in $\mathfrak{sl}^{\oplus3}\cong\mathfrak{so}(4)\oplus\mathfrak{sl}(2)$ via $(A,B)\mapsto(A,B,A)$.

\subsubsection*{6.4. $(000010;\ 3)$}

\[\begin{array}{rcl}
\bigg\{ & x_{\eseven{1}{2}{2}{3}{2}{1}{0}}+x_{\eseven{1}{1}{2}{3}{2}{1}{1}}+x_{\eseven{1}{1}{2}{2}{2}{2}{1}},\qquad h_{\eseven{1}{2}{2}{3}{2}{1}{0}}+h_{\eseven{1}{1}{2}{3}{2}{1}{1}}+h_{\eseven{1}{1}{2}{2}{2}{2}{1}},\\ & y_{\eseven{1}{2}{2}{3}{2}{1}{0}}+y_{\eseven{1}{1}{2}{3}{2}{1}{1}}+y_{\eseven{1}{1}{2}{2}{2}{2}{1}} & \bigg\}
\end{array}\]
We can take $e=(\varphi_6\wedge\varphi_4\wedge\varphi_3 + \varphi_6\wedge\varphi_5\wedge\varphi_2 - \varphi_1)\otimes e'_1$. We have $\mathfrak l=\mathfrak k_h\cong\mathfrak{gl}(6)\oplus\mathfrak{gl}(1)$ 
and 
$\mathfrak k_e=\mathfrak l_e+\mathfrak m$, 
where
$\mathfrak l_e\cong\mathfrak{gl}(1)\oplus\mathfrak{sp}(6)$ and $\mathfrak m$ is the $\mathfrak l_e$-complement in $\mathfrak n$ of the unique $1$-dimensional $\mathfrak l_e$-submodule of $\mathfrak k(1)\cap\mathfrak{so}(V)$.

\subsubsection*{6.5. $(010010;\ 1)$}
 
\[\begin{array}{rcl}
\bigg\{ & x_{\eseven{1}{2}{2}{3}{2}{2}{1}}+x_{\eseven{1}{1}{2}{3}{3}{2}{1}}+y_{\eseven{1}{0}{0}{0}{0}{0}{0}},\qquad h_{\eseven{1}{2}{2}{3}{2}{2}{1}}+h_{\eseven{1}{1}{2}{3}{3}{2}{1}}-h_{\eseven{1}{0}{0}{0}{0}{0}{0}},\\ & y_{\eseven{1}{2}{2}{3}{2}{2}{1}}+y_{\eseven{1}{1}{2}{3}{3}{2}{1}}+x_{\eseven{1}{0}{0}{0}{0}{0}{0}} & \bigg\}
\end{array}\]
We can take $e=(\varphi_6\wedge\varphi_5\wedge\varphi_4 -\varphi_3)\otimes e'_1+\varphi_6\otimes e'_2$.
We have $\mathfrak l=\mathfrak k_h\cong\mathfrak{gl}(2)\oplus\mathfrak{gl}(4)\oplus\mathfrak{gl}(1)$ 
and 
$\mathfrak k_e=\mathfrak l_e+\mathfrak m$, 
where
$\mathfrak l_e\cong\mathfrak{gl}(1)\oplus\mathfrak{sl}(2)\oplus\mathfrak{sp}(4)$. The $\mathfrak l_e$ submodule $\mathfrak m\subset\mathfrak n$ can be described as follows. In $\mathfrak k(1)$ there are precisely two $1$-dimensional $\mathfrak l_e$-submodules lying into different isotypic $\mathfrak l$-components, call them $\mathfrak u_1$ and $\mathfrak u_2$: then $\mathfrak m$ is the sum of the $\mathfrak l_e$-complement of $\mathfrak u_1\oplus\mathfrak u_2$ in $\mathfrak n$, plus a one dimensional submodule $\mathfrak u\subset\mathfrak u_1\oplus\mathfrak u_2$ that projects nontrivially on both summands (which are isomorphic as $\mathfrak l_e$-modules).

\subsection{$\mathsf E_7/\mathsf E_6\times \mathbb C^\times$}

We have $\mathfrak k=\mathfrak e_6\oplus\mathfrak{gl}(1)$, $\mathfrak p=V(\omega_1)\oplus V(\omega_6)$ as $\mathfrak k^{\mathrm{ss}}$-module.
Take $\theta$ such that $\mathfrak t^\theta=\mathfrak t$ and take the following root vectors for the simple roots of $\mathfrak k$:
\[x_{\eseven{1}{0}{0}{0}{0}{0}{0}},\ x_{\eseven{0}{1}{0}{0}{0}{0}{0}},\ 
x_{\eseven{0}{0}{1}{0}{0}{0}{0}},\ x_{\eseven{0}{0}{0}{1}{0}{0}{0}},\
x_{\eseven{0}{0}{0}{0}{1}{0}{0}},\ x_{\eseven{0}{0}{0}{0}{0}{1}{0}}.\]
Then $x_{\eseven{2}{2}{3}{4}{3}{2}{1}}$ and $y_{\eseven{0}{0}{0}{0}{0}{0}{1}}$ are highest weight vectors in $\mathfrak p$.

\subsubsection*{7.1. $(100000;\ 0)$}

\[\left\{x_{\eseven{2}{2}{3}{4}{3}{2}{1}},\ h_{\eseven{2}{2}{3}{4}{3}{2}{1}},\ y_{\eseven{2}{2}{3}{4}{3}{2}{1}}\right\}\]
We have $\mathfrak l=\mathfrak k_h\cong\mathfrak{so}(10)\oplus\mathfrak{gl}(1)^{\oplus2}$ 
and 
$\mathfrak k_e=\mathfrak l_e+\mathfrak n$, 
where
$\mathfrak l_e\cong\mathfrak{gl}(1)\oplus\mathfrak{so}(10)$.

\subsubsection*{7.2. $(000001;\ -2)$}

\[\left\{y_{\eseven{0}{0}{0}{0}{0}{0}{1}},\ -h_{\eseven{0}{0}{0}{0}{0}{0}{1}},\ x_{\eseven{0}{0}{0}{0}{0}{0}{1}}\right\}\]
This case is equal to the previous one up to an external automorphism of $\mathfrak k$.

\subsubsection*{7.3. $(000001;\ 0)$}

\[\left\{x_{\eseven{1}{2}{2}{3}{2}{2}{1}}+x_{\eseven{1}{1}{2}{3}{3}{2}{1}},\ h_{\eseven{1}{2}{2}{3}{2}{2}{1}}+h_{\eseven{1}{1}{2}{3}{3}{2}{1}},\ y_{\eseven{1}{2}{2}{3}{2}{2}{1}}+y_{\eseven{1}{1}{2}{3}{3}{2}{1}}\right\}\]
We have $\mathfrak l=\mathfrak k_h\cong\mathfrak{so}(10)\oplus\mathfrak{gl}(1)^{\oplus2}$ 
and 
$\mathfrak k_e=\mathfrak l_e+\mathfrak n$, 
where
$\mathfrak l_e\cong\mathfrak{gl}(1)\oplus\mathfrak{so}(9)$.

\subsubsection*{7.4. $(100000;\ -2)$}

\[\left\{y_{\eseven{0}{1}{0}{1}{1}{1}{1}}+y_{\eseven{0}{0}{1}{1}{1}{1}{1}},\ -h_{\eseven{0}{1}{0}{1}{1}{1}{1}}-h_{\eseven{0}{0}{1}{1}{1}{1}{1}},\ x_{\eseven{0}{1}{0}{1}{1}{1}{1}}+x_{\eseven{0}{0}{1}{1}{1}{1}{1}}\right\}\]
This case is equal to the previous one up to an external automorphism of $\mathfrak k$.

\subsubsection*{7.5. $(100001;\ -2)$}

\[\left\{x_{\eseven{2}{2}{3}{4}{3}{2}{1}}+y_{\eseven{0}{0}{0}{0}{0}{0}{1}},\ h_{\eseven{2}{2}{3}{4}{3}{2}{1}}-h_{\eseven{0}{0}{0}{0}{0}{0}{1}},\ y_{\eseven{2}{2}{3}{4}{3}{2}{1}}+x_{\eseven{0}{0}{0}{0}{0}{0}{1}}\right\}\]
We have $\mathfrak l=\mathfrak k_h\cong\mathfrak{so}(8)\oplus\mathfrak{gl}(1)^{\oplus3}$ 
and 
$\mathfrak k_e=\mathfrak l_e+\mathfrak n$, 
where
$\mathfrak l_e\cong\mathfrak{gl}(1)\oplus\mathfrak{so}(8)$.

\subsubsection*{7.6. $(000000;\ 2)$}

\[\begin{array}{rcl}
\bigg\{ & x_{\eseven{1}{1}{2}{2}{1}{1}{1}}+x_{\eseven{1}{1}{1}{2}{2}{1}{1}}+x_{\eseven{0}{1}{1}{2}{2}{2}{1}},\qquad h_{\eseven{1}{1}{2}{2}{1}{1}{1}}+h_{\eseven{1}{1}{1}{2}{2}{1}{1}}+h_{\eseven{0}{1}{1}{2}{2}{2}{1}},\\ & y_{\eseven{1}{1}{2}{2}{1}{1}{1}}+y_{\eseven{1}{1}{1}{2}{2}{1}{1}}+y_{\eseven{0}{1}{1}{2}{2}{2}{1}} & \bigg\}
\end{array}\]
We have $\mathfrak l=\mathfrak k_h=\mathfrak k$ 
and 
$\mathfrak k_e\cong\mathfrak f_4$.

\subsubsection*{7.7. $(000000;\ -2)$}

\[\begin{array}{rcl}
\bigg\{ & y_{\eseven{1}{1}{2}{2}{1}{1}{1}}+y_{\eseven{1}{1}{1}{2}{2}{1}{1}}+y_{\eseven{0}{1}{1}{2}{2}{2}{1}},\qquad -h_{\eseven{1}{1}{2}{2}{1}{1}{1}}-h_{\eseven{1}{1}{1}{2}{2}{1}{1}}-h_{\eseven{0}{1}{1}{2}{2}{2}{1}},\\ & x_{\eseven{1}{1}{2}{2}{1}{1}{1}}+x_{\eseven{1}{1}{1}{2}{2}{1}{1}}+x_{\eseven{0}{1}{1}{2}{2}{2}{1}} & \bigg\}
\end{array}\]
This case is equal to the previous one: the orbit can be obtained from the previous one by applying an external automorphism of $\mathfrak k$, but the centralizer is the same.

\subsubsection*{7.8. $(000002;\ -2)$}

\[\begin{array}{rcl}
\bigg\{ & x_{\eseven{1}{2}{2}{3}{2}{2}{1}}+x_{\eseven{1}{1}{2}{3}{3}{2}{1}}+y_{\eseven{0}{0}{0}{0}{0}{0}{1}},\qquad h_{\eseven{1}{2}{2}{3}{2}{2}{1}}+h_{\eseven{1}{1}{2}{3}{3}{2}{1}}-h_{\eseven{0}{0}{0}{0}{0}{0}{1}},\\ & y_{\eseven{1}{2}{2}{3}{2}{2}{1}}+y_{\eseven{1}{1}{2}{3}{3}{2}{1}}+x_{\eseven{0}{0}{0}{0}{0}{0}{1}} & \bigg\}
\end{array}\]
We have $\mathfrak l=\mathfrak k_h\cong\mathfrak{so}(10)\oplus\mathfrak{gl}(1)^{\oplus2}$ 
and 
$\mathfrak k_e=\mathfrak l_e+\mathfrak n$, 
where
$\mathfrak l_e\cong\mathfrak{so}(9)$.

\subsubsection*{7.9. $(200000;\ -2)$}

\[\begin{array}{rcl}
\bigg\{ & x_{\eseven{2}{2}{3}{4}{3}{2}{1}}+y_{\eseven{0}{1}{0}{1}{1}{1}{1}}+y_{\eseven{0}{0}{1}{1}{1}{1}{1}},\qquad h_{\eseven{2}{2}{3}{4}{3}{2}{1}}-h_{\eseven{0}{1}{0}{1}{1}{1}{1}}-h_{\eseven{0}{0}{1}{1}{1}{1}{1}},\\ 
 & y_{\eseven{2}{2}{3}{4}{3}{2}{1}}+x_{\eseven{0}{1}{0}{1}{1}{1}{1}}+x_{\eseven{0}{0}{1}{1}{1}{1}{1}} & \bigg\}
\end{array}\]
This case is equal to the previous one up to an external automorphism of $\mathfrak k$.

\subsubsection*{7.10. $(020000;\ -2)$}

\[\left\{x_{\eseven{2}{2}{3}{4}{3}{2}{1}}+y_{\eseven{1}{0}{1}{1}{1}{1}{1}},\ 2h_{\eseven{2}{2}{3}{4}{3}{2}{1}}-2h_{\eseven{1}{0}{1}{1}{1}{1}{1}},\ 2y_{\eseven{2}{2}{3}{4}{3}{2}{1}}+2x_{\eseven{1}{0}{1}{1}{1}{1}{1}}\right\}\]
We have $\mathfrak l=\mathfrak k_h\cong\mathfrak{sl}(6)\oplus\mathfrak{gl}(1)^{\oplus2}$ 
and 
$\mathfrak k_e=\mathfrak l_e+\mathfrak n$, 
where
$\mathfrak l_e\cong\mathfrak{gl}(1)\oplus\mathfrak{sl}(5)$.

\subsubsection*{7.11. $(010010;\ -2)$}

\[\begin{array}{rcl}
\bigg\{ & x_{\eseven{1}{2}{2}{3}{2}{2}{1}}+x_{\eseven{1}{1}{2}{3}{3}{2}{1}}+y_{\eseven{0}{0}{0}{0}{0}{1}{1}},\qquad 2h_{\eseven{1}{2}{2}{3}{2}{2}{1}}+h_{\eseven{1}{1}{2}{3}{3}{2}{1}}-2h_{\eseven{0}{0}{0}{0}{0}{1}{1}},\\ & 2y_{\eseven{1}{2}{2}{3}{2}{2}{1}}+y_{\eseven{1}{1}{2}{3}{3}{2}{1}}+2x_{\eseven{0}{0}{0}{0}{0}{1}{1}} & \bigg\}
\end{array}\]
We have $\mathfrak l=\mathfrak k_h\cong\mathfrak{sl}(4)\oplus\mathfrak{sl}(2)\oplus\mathfrak{gl}(1)^{\oplus3}$ 
and 
$\mathfrak k_e=\mathfrak l_e+\mathfrak m$, 
where
$\mathfrak l_e\cong\mathfrak{gl}(1)\oplus\mathfrak{sl}(4)$. The $\mathfrak l_e$-submodule $\mathfrak m\subset\mathfrak n$ can be described as follows.

Let $\mathfrak u_1$ and $\mathfrak u_2$ be the $\mathfrak{sl}(4)$-submodules of $\mathfrak k(1)$ with lowest weight vectors $x_{\alpha_2}$ and $x_{\alpha_5}$, respectively. They are isomorphic as $\mathfrak l_e$-modules, but they lie into distinct isotypic $\mathfrak l$-components. Then $\mathfrak m$ is the sum of the $\mathfrak l_e$-complement of $\mathfrak u_1\oplus\mathfrak u_2$ in $\mathfrak n$, plus a simple $\mathfrak l_e$-submodule $\mathfrak u\subset\mathfrak u_1\oplus\mathfrak u_2$ that projects nontrivially on both summands.

\subsubsection*{7.12. $(011000;\ -3)$}

\[\begin{array}{rcl}
\bigg\{ & x_{\eseven{1}{2}{3}{4}{3}{2}{1}}+y_{\eseven{0}{1}{0}{1}{1}{1}{1}}+y_{\eseven{0}{0}{1}{1}{1}{1}{1}},\qquad 2h_{\eseven{1}{2}{3}{4}{3}{2}{1}}-h_{\eseven{0}{1}{0}{1}{1}{1}{1}}-2h_{\eseven{0}{0}{1}{1}{1}{1}{1}},\\ 
 & 2y_{\eseven{1}{2}{3}{4}{3}{2}{1}}+x_{\eseven{0}{1}{0}{1}{1}{1}{1}}+2x_{\eseven{0}{0}{1}{1}{1}{1}{1}} & \bigg\}
\end{array}\]
This case is equal to the previous one up to an external automorphism of $\mathfrak k$.

\subsection{$\mathsf E_8/\mathsf D_8$}

We have $\mathfrak k=\mathfrak{so}(16)$, $\mathfrak p=V(\omega_7)$.
Take $\theta$ such that $\mathfrak t^\theta=\mathfrak t$ and take the following root vectors for the simple roots of $\mathfrak k$:
\[x_{\eeight{2}{2}{3}{4}{3}{2}{1}{0}},\ x_{\eeight{0}{0}{0}{0}{0}{0}{0}{1}},\ 
x_{\eeight{0}{0}{0}{0}{0}{0}{1}{0}},\ x_{\eeight{0}{0}{0}{0}{0}{1}{0}{0}},\
x_{\eeight{0}{0}{0}{0}{1}{0}{0}{0}},\ x_{\eeight{0}{0}{0}{1}{0}{0}{0}{0}},\ 
x_{\eeight{0}{1}{0}{0}{0}{0}{0}{0}},\ x_{\eeight{0}{0}{0}{1}{0}{0}{0}{0}}.\]
Then $x_{\eeight{1}{3}{3}{5}{4}{3}{2}{1}}$ is a highest weight vector in $\mathfrak p$.

Fix a basis $e_1,\ldots,e_8$ of $U$ and a dual basis $\varphi_1,\ldots,\varphi_8$ of $U^*$, so that we have a nondegenerate symmetric bilinear form on $V=U\oplus U^*$, and $\mathfrak k=\mathfrak{so}(V)$. The spin representations can be realized in $\mathsf\Lambda U^*$ (as the odd and even degree parts) via the anti-symmetric square 
\[\sigma^2\colon \mathsf\Lambda^2V\otimes\mathsf\Lambda U^*\to\mathsf\Lambda U^*\] 
of the contraction-extension map
\[\sigma\colon V\otimes\mathsf\Lambda U^*\to\mathsf\Lambda U^*\]
and the natural isomorphism between $\mathfrak{so}(V)$ and $\mathsf\Lambda^2V$. We can identify $x_{\eeight{1}{3}{3}{5}{4}{3}{2}{1}}=\varphi_8$ in $\mathfrak p=\mathsf\Lambda^{\mathrm{odd}} U^*$. 

\subsubsection*{8.1. $(00000010)$}
 
\[\left\{x_{\eeight{1}{3}{3}{5}{4}{3}{2}{1}},\ h_{\eeight{1}{3}{3}{5}{4}{3}{2}{1}},\ y_{\eeight{1}{3}{3}{5}{4}{3}{2}{1}}\right\}\]
We can take $e=\varphi_8$. We have $\mathfrak l=\mathfrak k_h\cong\mathfrak{gl}(8)$ 
and 
$\mathfrak k_e=\mathfrak l_e+\mathfrak n$, 
where
$\mathfrak l_e\cong\mathfrak{sl}(8)$.

\subsubsection*{8.2. $(00010000)$}
 
\[\left\{x_{\eeight{1}{2}{3}{4}{3}{3}{2}{1}}+x_{\eeight{1}{2}{2}{4}{4}{3}{2}{1}},\ h_{\eeight{1}{2}{3}{4}{3}{3}{2}{1}}+h_{\eeight{1}{2}{2}{4}{4}{3}{2}{1}},\ y_{\eeight{1}{3}{2}{4}{3}{3}{2}{1}}+y_{\eeight{1}{2}{2}{4}{4}{3}{2}{1}}\right\}\]
We can take $e=\varphi_8\wedge\varphi_7\wedge\varphi_6-\varphi_5$. We have $\mathfrak l=\mathfrak k_h\cong\mathfrak{gl}(4)\oplus\mathfrak{so}(8)$ 
and 
$\mathfrak k_e=\mathfrak l_e+\mathfrak n$, 
where
$\mathfrak l_e\cong\mathfrak{sl}(4)\oplus\mathfrak{so}(7)$.

\subsubsection*{8.3. $(01000010)$}
 
\[\begin{array}{rcl}
\bigg\{ & x_{\eeight{1}{2}{3}{4}{3}{2}{1}{1}}+x_{\eeight{1}{2}{2}{4}{3}{2}{2}{1}}+x_{\eeight{1}{2}{2}{3}{3}{3}{2}{1}},\qquad h_{\eeight{1}{2}{3}{4}{3}{2}{1}{1}}+h_{\eeight{1}{2}{2}{4}{3}{2}{2}{1}}+h_{\eeight{1}{2}{2}{3}{3}{3}{2}{1}},\\ 
 & y_{\eeight{1}{2}{3}{4}{3}{2}{1}{1}}+y_{\eeight{1}{2}{2}{4}{3}{2}{2}{1}}+y_{\eeight{1}{2}{2}{3}{3}{3}{2}{1}} & \bigg\}
\end{array}\]
We can take $e=\varphi_8\wedge\varphi_6\wedge\varphi_5 + \varphi_8\wedge\varphi_7\wedge\varphi_4 - \varphi_3$. We have $\mathfrak l=\mathfrak k_h\cong\mathfrak{gl}(2)\oplus\mathfrak{gl}(6)$ 
and 
$\mathfrak k_e=\mathfrak l_e+\mathfrak m$, 
where
$\mathfrak l_e\cong\mathfrak{gl}(1)\oplus\mathfrak{sl}(2)\oplus\mathfrak{sp}(6)$ and $\mathfrak m$ is the $\mathfrak l_e$-complement in $\mathfrak n$ of the unique $1$-dimensional $\mathfrak l_e$-submodule of $\mathfrak k(1)$.
 
\subsubsection*{8.6. $(10001000)$}
 
\[\begin{array}{rll}
\bigg\{ & x_{\eeight{1}{2}{3}{4}{3}{2}{1}{0}}+x_{\eeight{1}{2}{2}{4}{3}{2}{1}{1}}+x_{\eeight{1}{2}{2}{3}{3}{2}{2}{1}}+x_{\eeight{1}{1}{2}{3}{3}{3}{2}{1}},\\ 
\rule{0pt}{15pt} & h_{\eeight{1}{2}{3}{4}{3}{2}{1}{0}}+h_{\eeight{1}{2}{2}{4}{3}{2}{1}{1}}+h_{\eeight{1}{2}{2}{3}{3}{2}{2}{1}}+h_{\eeight{1}{1}{2}{3}{3}{3}{2}{1}},\\ 
 & y_{\eeight{1}{2}{3}{4}{3}{2}{1}{0}}+y_{\eeight{1}{2}{2}{4}{3}{2}{1}{1}}+y_{\eeight{1}{2}{2}{3}{3}{2}{2}{1}}+y_{\eeight{1}{1}{2}{3}{3}{3}{2}{1}} & \bigg\}
\end{array}\]
We can take $e=\varphi_7\wedge\varphi_6\wedge\varphi_5 + \varphi_8\wedge\varphi_6\wedge\varphi_4 + \varphi_8\wedge\varphi_7\wedge\varphi_3 - \varphi_2$. We have $\mathfrak l=\mathfrak k_h\cong\mathfrak{gl}(1)\oplus\mathfrak{gl}(4)\oplus\mathfrak{so}(6)$ 
and 
$\mathfrak k_e=\mathfrak l_e+\mathfrak m$, 
where
$\mathfrak l_e\cong\mathfrak{gl}(1)\oplus\mathfrak{sl}(4)$. Here $\mathfrak{sl}(4)$ is included in $\mathfrak{sl}(4)\oplus\mathfrak{so}(6)$ and contains $x_{\alpha_8}-x_{\alpha_3}$, $x_{\alpha_7}-x_{\alpha_4}$, $x_{\alpha_6}-x_{\alpha_2}$ as root vectors for its simple roots, whereas
$\mathfrak m$ is the $\mathfrak l_e$-complement in $\mathfrak n$ of the unique $4$-dimensional $\mathfrak l_e$-submodule in the simple $\mathfrak l$-submodule of lowest weight vector $x_{\alpha_5}$.

%The $\mathfrak{sl}(4)$ summand, in $\mathfrak{sl}(4)\oplus\mathfrak{so}(6)$, contains $x_{\alpha_8}-x_{\alpha_3}$, $x_{\alpha_7}-x_{\alpha_4}$, $x_{\alpha_6}-x_{\alpha_2}$ as root vectors for its simple roots. There is a unique $\mathfrak l_e$-submodule of dimension 4 included in the simple $\mathfrak l$-submodule of lowest weight vector $x_{\alpha_5}$, call it $\mathfrak u$. As $\mathfrak l_e$-module, $\mathfrak m$ is the $\mathfrak l_e$-complement of $\mathfrak u$ in $\mathfrak n$.

\subsection{$\mathsf E_8/\mathsf E_7 \times \mathsf A_1$}

We have $\mathfrak k=\mathfrak e_7+\mathfrak{sl}(2)$, $\mathfrak p=V(\omega_7+\omega')$.
Take $\theta$ such that $\mathfrak t^\theta=\mathfrak t$ and take the following root vectors for the simple roots of $\mathfrak k$:
\[x_{\eeight{1}{0}{0}{0}{0}{0}{0}{0}},\ x_{\eeight{0}{1}{0}{0}{0}{0}{0}{0}},\ 
x_{\eeight{0}{0}{1}{0}{0}{0}{0}{0}},\ x_{\eeight{0}{0}{0}{1}{0}{0}{0}{0}},\
x_{\eeight{0}{0}{0}{0}{1}{0}{0}{0}},\ x_{\eeight{0}{0}{0}{0}{0}{1}{0}{0}},\ 
x_{\eeight{0}{0}{0}{0}{0}{0}{1}{0}},\ x_{\eeight{2}{3}{4}{6}{5}{4}{3}{2}}.\]
Then $x_{\eeight{2}{3}{4}{6}{5}{4}{3}{1}}$ is a highest weight vector in $\mathfrak p$.

\subsubsection*{9.1. $(0000001;\ 1)$}
 
\[\left\{x_{\eeight{2}{3}{4}{6}{5}{4}{3}{1}},\ h_{\eeight{2}{3}{4}{6}{5}{4}{3}{1}},\ y_{\eeight{2}{3}{4}{6}{5}{4}{3}{1}}\right\}\]
We have $\mathfrak l=\mathfrak k_h\cong\mathfrak e_6\oplus\mathfrak{gl}(1)^{\oplus2}$ 
and 
$\mathfrak k_e=\mathfrak l_e+\mathfrak n$, 
where
$\mathfrak l_e\cong\mathfrak{gl}(1)\oplus\mathfrak e_6$.

\subsubsection*{9.2. $(1000000;\ 2)$}
 
\[\left\{x_{\eeight{2}{3}{3}{5}{4}{3}{2}{1}}+x_{\eeight{2}{2}{4}{5}{4}{3}{2}{1}},\ h_{\eeight{2}{3}{3}{5}{4}{3}{2}{1}}+h_{\eeight{2}{2}{4}{5}{4}{3}{2}{1}},\ y_{\eeight{2}{3}{3}{5}{4}{3}{2}{1}}+y_{\eeight{2}{2}{4}{5}{4}{3}{2}{1}}\right\}\]
We have $\mathfrak l=\mathfrak k_h\cong\mathfrak{so}(12)\oplus\mathfrak{gl}(1)^{\oplus2}$ 
and 
$\mathfrak k_e=\mathfrak l_e+\mathfrak n$, 
where
$\mathfrak l_e\cong\mathfrak{gl}(1)\oplus\mathfrak{so}(11)$.

\subsubsection*{9.3. $(0000010;\ 0)$}
 
\[\left\{x_{\eeight{2}{3}{4}{6}{5}{4}{2}{1}}+y_{\eeight{0}{0}{0}{0}{0}{0}{0}{1}},\ h_{\eeight{2}{3}{4}{6}{5}{4}{2}{1}}-h_{\eeight{0}{0}{0}{0}{0}{0}{0}{1}},\ y_{\eeight{2}{3}{4}{6}{5}{4}{2}{1}}+x_{\eeight{0}{0}{0}{0}{0}{0}{0}{1}}\right\}\]
We have $\mathfrak l=\mathfrak k_h\cong\mathfrak{so}(10)\oplus\mathfrak{sl}(2)^{\oplus2}\oplus\mathfrak{gl}(1)$ 
and 
$\mathfrak k_e=\mathfrak l_e+\mathfrak n$, 
where
$\mathfrak l_e\cong\mathfrak{so}(10)\oplus\mathfrak{sl}(2)$ and $\mathfrak{sl}(2)$ is embedded diagonally in $\mathfrak{sl}(2)^{\oplus2}$.

\subsubsection*{9.4. $(0000001;\ 3)$}
 
\[\begin{array}{rcl}
\bigg\{ & x_{\eeight{2}{2}{3}{4}{3}{2}{2}{1}}+x_{\eeight{1}{2}{3}{4}{3}{3}{2}{1}}+x_{\eeight{1}{2}{2}{4}{4}{3}{2}{1}},\qquad h_{\eeight{2}{2}{3}{4}{3}{2}{2}{1}}+h_{\eeight{1}{2}{3}{4}{3}{3}{2}{1}}+h_{\eeight{1}{2}{2}{4}{4}{3}{2}{1}},\\ 
 & y_{\eeight{2}{2}{3}{4}{3}{2}{2}{1}}+y_{\eeight{1}{2}{3}{4}{3}{3}{2}{1}}+y_{\eeight{1}{2}{2}{4}{4}{3}{2}{1}} & \bigg\}
\end{array}\]
We have $\mathfrak l=\mathfrak k_h\cong\mathfrak e_6\oplus\mathfrak{gl}(1)^{\oplus2}$ 
and 
$\mathfrak k_e=\mathfrak l_e+\mathfrak m$, 
where
$\mathfrak l_e\cong\mathfrak{gl}(1)\oplus\mathfrak f_4$ and $\mathfrak m$ is the $\mathfrak l_e$-complement of $\mathfrak u$ in $\mathfrak n$: $\mathfrak u$ being the unique $1$-dimensional $\mathfrak f_4$-submodule in the simple $\mathfrak e_6$-submodule of $\mathfrak k(1)$ of lowest weight vector $x_{\alpha_7}$.

%In the simple $\mathfrak e_6$-submodule of $\mathfrak k(1)$ of lowest weight vector $x_{\alpha_7}$  there is a unique $\mathfrak f_4$-submodule of dimension 1, call it $\mathfrak u$. As $\mathfrak l_e$-module, $\mathfrak m$ is the $\mathfrak l_e$-complement of $\mathfrak u$ in $\mathfrak n$.

\subsubsection*{9.5. $(1000001;\ 1)$}

\[\begin{array}{rcl}
\bigg\{ & x_{\eeight{2}{3}{3}{5}{4}{3}{2}{1}}+x_{\eeight{2}{2}{4}{5}{4}{3}{2}{1}}+y_{\eeight{0}{0}{0}{0}{0}{0}{0}{1}},\qquad h_{\eeight{2}{3}{3}{5}{4}{3}{2}{1}}+h_{\eeight{2}{2}{4}{5}{4}{3}{2}{1}}-h_{\eeight{0}{0}{0}{0}{0}{0}{0}{1}},\\ 
 & y_{\eeight{2}{3}{3}{5}{4}{3}{2}{1}}+y_{\eeight{2}{2}{4}{5}{4}{3}{2}{1}}+x_{\eeight{0}{0}{0}{0}{0}{0}{0}{1}} & \bigg\}
\end{array}\]
We have $\mathfrak l=\mathfrak k_h\cong\mathfrak{so}(10)\oplus\mathfrak{gl}(1)^{\oplus3}$ 
and 
$\mathfrak k_e=\mathfrak l_e+\mathfrak m$, 
where
$\mathfrak l_e\cong\mathfrak{gl}(1)\oplus\mathfrak{so}(9)$. The $\mathfrak l_e$-submodule $\mathfrak m\subset\mathfrak n$ can be described as follows.

Let $\mathfrak u_1$ be the unique $1$-dimensional $\mathfrak l_e$-submodule of the simple $\mathfrak l$-submodule of $\mathfrak k(1)$ of lowest weight vector $x_{\alpha_7}$, and let $\mathfrak u_2$ be the $1$-dimensional $\mathfrak l$-submodule of $\mathfrak k(1)$ spanned by $x_{\eeight{2}{3}{4}{6}{5}{4}{3}{2}}$. Then $\mathfrak m$ is the sum of the $\mathfrak l_e$-complement of $\mathfrak u_1\oplus\mathfrak u_2$ in $\mathfrak n$, plus a one dimensional submodule $\mathfrak u\subset\mathfrak u_1\oplus\mathfrak u_2$ that projects nontrivially on both summands (which are isomorphic as $\mathfrak l_e$-modules).

\subsection{$\mathsf F_4/\mathsf C_3 \times \mathsf A_1$}

We have $\mathfrak k=\mathfrak{sp}(6)+\mathfrak{sl}(2)$, $\mathfrak p=V(\omega_3+\omega')$.
Take $\theta$ such that $\mathfrak t^\theta=\mathfrak t$ and take the following root vectors for the simple roots of $\mathfrak k$:
\[x_{\ffour{0}{0}{0}{1}},\ x_{\ffour{0}{0}{1}{0}},\ 
x_{\ffour{0}{1}{0}{0}},\ x_{\ffour{2}{3}{4}{2}}.\]
Then $x_{\ffour{1}{3}{4}{2}}$ is a highest weight vector in $\mathfrak p$.

Fix a basis $e_1,e_2,e_3,e_{-3},e_{-2},e_{-1}$ of $V$ and a skew-symmetric bilinear form $\omega$ on $V$ such that $\omega(e_i,e_j)=\delta_{i,-j}$. Fix also a basis $e'_1,e'_2$ of $W$. We have $\mathfrak k=\mathfrak{sp}(V,\omega)\oplus\mathfrak{sl}(W)$. We can identify $x_{\ffour{1}{3}{4}{2}}=e_1\wedge e_2\wedge e_3\otimes e'_1$ in $\mathfrak p\subset\left(\mathsf\Lambda^3V\right)\otimes W$. 

\subsubsection*{10.1. $(001;\ 1)$}
 
\[\left\{x_{\ffour{1}{3}{4}{2}},\ h_{\ffour{1}{3}{4}{2}},\ y_{\ffour{1}{3}{4}{2}}\right\}\]
We can take $e=e_1\wedge e_2\wedge e_3\otimes e'_1$. We have $\mathfrak l=\mathfrak k_h\cong\mathfrak{gl}(3)\oplus\mathfrak{gl}(1)$ 
and 
$\mathfrak k_e=\mathfrak l_e+\mathfrak n$, 
where
$\mathfrak l_e\cong\mathfrak{gl}(1)\oplus\mathfrak{sl}(3)$.

\subsubsection*{10.2. $(100;\ 2)$}
 
\[\left\{x_{\ffour{1}{2}{3}{2}},\ h_{\ffour{1}{2}{3}{2}},\ y_{\ffour{1}{2}{3}{2}}\right\}\]
We can take $e=(e_1\wedge e_2\wedge e_{-2}-e_1\wedge e_3\wedge e_{-3})\otimes e'_1$. We have $\mathfrak l=\mathfrak k_h\cong\mathfrak{gl}(1)\oplus\mathfrak{sp}(4)\oplus\mathfrak{gl}(1)$ 
and 
$\mathfrak k_e=\mathfrak l_e+\mathfrak n$, 
where
$\mathfrak l_e\cong\mathfrak{gl}(1)\oplus\mathfrak{sl}(2)\oplus\mathfrak{sl}(2)$.

\subsubsection*{10.3. $(010;\ 0)$}
 
\[\left\{x_{\ffour{1}{2}{4}{2}}+y_{\ffour{1}{0}{0}{0}},\ h_{\ffour{1}{2}{4}{2}}-h_{\ffour{1}{0}{0}{0}},\ y_{\ffour{1}{2}{4}{2}}+x_{\ffour{1}{0}{0}{0}}\right\}\]
We can take $e=e_1\wedge e_2\wedge e_3\otimes e'_2-e_1\wedge e_2\wedge e_{-3}\otimes e'_1$. We have $\mathfrak l=\mathfrak k_h\cong\mathfrak{gl}(2)\oplus\mathfrak{sl}(2)\oplus\mathfrak{sl}(2)$ 
and 
$\mathfrak k_e=\mathfrak l_e+\mathfrak n$, 
where
$\mathfrak l_e\cong\mathfrak{sl}(2)\oplus\mathfrak{sl}(2)$ embedded diagonally in $\mathfrak{sl}(2)^{\oplus3}$ via $(A,B)\mapsto(A,B,B)$.

\subsubsection*{10.4. $(001;\ 3)$}
 
\[\left\{x_{\ffour{1}{2}{3}{1}}+x_{\ffour{1}{2}{2}{2}},\ h_{\ffour{1}{2}{3}{1}}+h_{\ffour{1}{2}{2}{2}},\ y_{\ffour{1}{2}{3}{1}}+y_{\ffour{1}{2}{2}{2}}\right\}\]
We can take $e=(e_1\wedge e_2\wedge e_{-1}-e_1\wedge e_3\wedge e_{-2}+e_2\wedge e_3\wedge e_{-3})\otimes e'_1$. We have $\mathfrak l=\mathfrak k_h\cong\mathfrak{gl}(3)\oplus\mathfrak{gl}(1)$ 
and 
$\mathfrak k_e=\mathfrak l_e+\mathfrak m$, 
where
$\mathfrak l_e\cong\mathfrak{gl}(1)\oplus\mathfrak{so}(3)$ and 
$\mathfrak m$ is the $\mathfrak l_e$-complement of $\mathfrak u$ in $\mathfrak n$: $\mathfrak u$ being the unique $1$-dimensional $\mathfrak l_e$-submodule of $\mathfrak k(1)\cap\mathfrak{sp}(V,\omega)$.

\subsubsection*{10.5. $(101;\ 1)$}
 
\[\left\{x_{\ffour{1}{2}{3}{2}}+y_{\ffour{1}{0}{0}{0}},\ h_{\ffour{1}{2}{3}{2}}-h_{\ffour{1}{0}{0}{0}},\ y_{\ffour{1}{2}{3}{2}}+x_{\ffour{1}{0}{0}{0}}\right\}\]
We can take $e=e_1\wedge e_2\wedge e_3\otimes e'_2+(e_1\wedge e_2\wedge e_{-2}-e_1\wedge e_3\wedge e_{-3})\otimes e'_1$. We have $\mathfrak l=\mathfrak k_h\cong\mathfrak{gl}(1)\oplus\mathfrak{gl}(2)\oplus\mathfrak{gl}(1)$ 
and 
$\mathfrak k_e=\mathfrak l_e+\mathfrak m$, 
where
$\mathfrak l_e\cong\mathfrak{gl}(1)^{\oplus2}$. The $\mathfrak l_e$-submodule $\mathfrak m\subset\mathfrak n$ can be described as follows. 

Let $\mathfrak u_1$ be the $1$-dimensional $\mathfrak l_e$-submodule spanned by $x_{\alpha_2+\alpha_3}$, which is contained in the simple $\mathfrak l$-submodule of $\mathfrak k(1)$ of lowest weight vector $x_{\alpha_2}$. 
Let $\mathfrak u_2$ be the $1$-dimensional $\mathfrak l$-submodule spanned by $x_{\ffour{2}{3}{4}{2}}$, it is equal $\mathfrak k(1)\cap\mathfrak{sl}(W)$. 
Then $\mathfrak m$ is the sum of the $\mathfrak l_e$-complement of $\mathfrak u_1\oplus\mathfrak u_2$ in $\mathfrak n$, plus a one dimensional submodule $\mathfrak u\subset\mathfrak u_1\oplus\mathfrak u_2$ that projects nontrivially on both summands (which are isomorphic as $\mathfrak l_e$-modules).

%In the simple $\mathfrak l$-submodule of $\mathfrak k(1)$ of lowest weight vector $x_{\alpha_2}$ take the $1$ dimensional $\mathfrak l_e$-submodule spanned by $x_{\alpha_2+\alpha_3}$, call it $\mathfrak u_1$. Take also $\mathfrak k(1)\cap\mathfrak{sl}(W)$, which is the $1$ dimensional $\mathfrak l$-submodule  spanned by $x_{\ffour{2}{3}{4}{2}}$, and call it $\mathfrak u_2$. 

%As $\mathfrak l_e$-module, $\mathfrak m$ is the $\mathfrak l_e$-complement of $\mathfrak u$ in $\mathfrak n$, where $\mathfrak u\subset\mathfrak u_1\oplus\mathfrak u_2$ projects nontrivially on both summands.

\subsection{$\mathsf F_4/\mathsf B_4$}

We have $\mathfrak k=\mathfrak{so}(9)$, $\mathfrak p=V(\omega_4)$.
Take $\theta$ such that $\mathfrak t^\theta=\mathfrak t$ and take the following root vectors for the simple roots of $\mathfrak k$:
\[x_{\ffour{0}{1}{2}{2}},\ x_{\ffour{1}{0}{0}{0}},\ 
x_{\ffour{0}{1}{0}{0}},\ x_{\ffour{0}{0}{1}{0}}.\]
Then $x_{\ffour{1}{2}{3}{1}}$ is a highest weight vector in $\mathfrak p$.

Fix a basis $e_1,e_2,e_3,e_4$ of $U$ and a dual basis $\varphi_1,\varphi_2,\varphi_3,\varphi_4$ of $U^*$, so that we have a nondegenerate symmetric bilinear form on $U\oplus U^*$, extend it to $V=U\oplus \mathbb Ce_0 \oplus U^*$ by setting $e_0$ orthogonal to $U\oplus U^*$ and $(e_0,e_0)=1$. We have $\mathfrak k=\mathfrak{so}(V)$. The spin representation can be realized in $\mathsf\Lambda U^*$ via the anti-symmetric square 
\[\sigma^2\colon \mathsf\Lambda^2V\otimes\mathsf\Lambda U^*\to\mathsf\Lambda U^*\] 
of the contraction-extension map
\[\sigma\colon (U\oplus U^*)\otimes\mathsf\Lambda U^*\to\mathsf\Lambda U^*\]
extended with the signed identity on $\mathbb Ce_0$
\[e_0\otimes \varphi\mapsto (-1)^{\deg\varphi}\varphi\]
and the natural isomorphism between $\mathfrak{so}(V)$ and $\mathsf\Lambda^2V$. We can identify $x_{\ffour{1}{2}{3}{1}}=1$ in $\mathfrak p=\mathsf\Lambda U^*$. 

\subsubsection*{11.1. $(0001)$}
 
\[\left\{x_{\ffour{1}{2}{3}{1}},\ h_{\ffour{1}{2}{3}{1}},\ y_{\ffour{1}{2}{3}{1}}\right\}\]
We can take $e=1$. We have $\mathfrak l=\mathfrak k_h\cong\mathfrak{gl}(4)$ 
and 
$\mathfrak k_e=\mathfrak l_e+\mathfrak n$, 
where
$\mathfrak l_e\cong\mathfrak{sl}(4)$.

\subsubsection*{11.2. $(4000)$}
 
\[\left\{x_{\ffour{1}{1}{1}{1}}+x_{\ffour{1}{1}{1}{1}},\ 2h_{\ffour{1}{1}{1}{1}}+2h_{\ffour{1}{1}{1}{1}},\ 2y_{\ffour{1}{1}{1}{1}}+2y_{\ffour{1}{1}{1}{1}}\right\}\]
We can take $e=\varphi_4\wedge\varphi_3+\varphi_2$. We have $\mathfrak l=\mathfrak k_h\cong\mathfrak{gl}(1)\oplus\mathfrak{so}(7)$ 
and 
$\mathfrak k_e=\mathfrak l_e+\mathfrak n$, 
where
$\mathfrak l_e\cong\mathfrak g_2$.

\subsection{$\mathsf G_2/\mathsf A_1 \times \mathsf A_1$}

We have $\mathfrak k=\mathfrak{sl}(2)+\mathfrak{sl}(2)$, $\mathfrak p=V(3\omega+\omega')$.
Take $\theta$ such that $\mathfrak t^\theta=\mathfrak t$ and take $x_{\alpha_1}$ and $x_{3\alpha_1+2\alpha_2}$ as root vectors for the simple roots of $\mathfrak k$.
Then $x_{3\alpha_1+\alpha_2}$ is a highest weight vector in $\mathfrak p$.

Fix a basis $e_1,e_2$ of $V$ and a basis $e'_1,e'_2$ of $W$. We have $\mathfrak k=\mathfrak{sl}(V)\oplus\mathfrak{sl}(W)$. We can identify $x_{3\alpha_1+\alpha_2}=(e_1)^3\otimes e'_1$ in $\mathfrak p=\left(\mathsf S^3 V\right)\otimes W$. 

\subsubsection*{12.1. $(1;\ 1)$}
 
\[\left\{x_{3\alpha_1+\alpha_2},\ h_{3\alpha_1+\alpha_2},\ y_{3\alpha_1+\alpha_2}\right\}\]
We can take $e=(e_1)^3\otimes e'_1$. We have $\mathfrak l=\mathfrak k_h\cong\mathfrak{gl}(1)\oplus\mathfrak{gl}(1)$ 
and 
$\mathfrak k_e=\mathfrak l_e+\mathfrak n$, 
where
$\mathfrak l_e\cong\mathfrak{gl}(1)$.

\subsubsection*{12.2. $(1;\ 3)$}
 
\[\left\{x_{2\alpha_1+\alpha_2},\ h_{2\alpha_1+\alpha_2},\ y_{2\alpha_1+\alpha_2}\right\}\]
We can take $e=(e_1)^2e_2\otimes e'_1$. We have $\mathfrak l=\mathfrak k_h\cong\mathfrak{gl}(1)\oplus\mathfrak{gl}(1)$ 
and 
$\mathfrak k_e=\mathfrak l_e+\mathfrak m$, 
where
$\mathfrak l_e\cong\mathfrak{gl}(1)$ and $\mathfrak m$ is the $\mathfrak l_e$-complement in $\mathfrak n$ of the $1$-dimensional $\mathfrak l$-module $\mathfrak k(1)\cap\mathfrak{sl}(V)$.

%%%%%%%%%%%%%%%%%%%%%%%%%%%%%%%%%%%%%%%%%%%%%%%%%%%%%%%%%%%%%%%%%%%%%%%%%%%%%%%%%%%%%%%%%%%%%%%%%%%%%%%%%%%%%%%%%%%%%%%

\section{Tables of spherical nilpotent $K$-orbits  in $\mathfrak p$ in the exceptional cases}\label{B}

Here we denote by $\alpha_1,\ldots,\alpha_n,\alpha'_1,\ldots,\alpha'_{n'},\ldots$ the simple roots of $K$, enumerated as in Bourbaki, and by $\omega_1,\ldots,\omega_n,\omega'_1,\ldots,\omega'_{n'},\ldots$ the corresponding fundamental weights. Recall that, when the symmetric pair $(G,K)$ is of Hermitian type, we denote by $\chi$ the central character of $K$ on $\mathfrak p_+$ (see Section~\ref{s:3}).

For every spherical nilpotent orbit $Ke$ in $\mathfrak p$, we report the Kostant-Dynkin diagram and the height of $Ge$ (columns~2 and 3), the Kostant-Dynkin diagram and the $\mathfrak p$-height of $Ke$ (columns~4 and 5), the Luna diagram and the set of spherical roots of the spherical system of $K\pi(e)$ (columns~6 and 7), the normality of $\ol{Ke}$ (column~8), the codimension of $\ol{Ke} \smallsetminus Ke$ in $\ol{Ke}$ (column~9) and the generators of the weight semigroup of $\wt{Ke}$ (column~10).

\includepdf[fitpaper,pages=-,landscape]{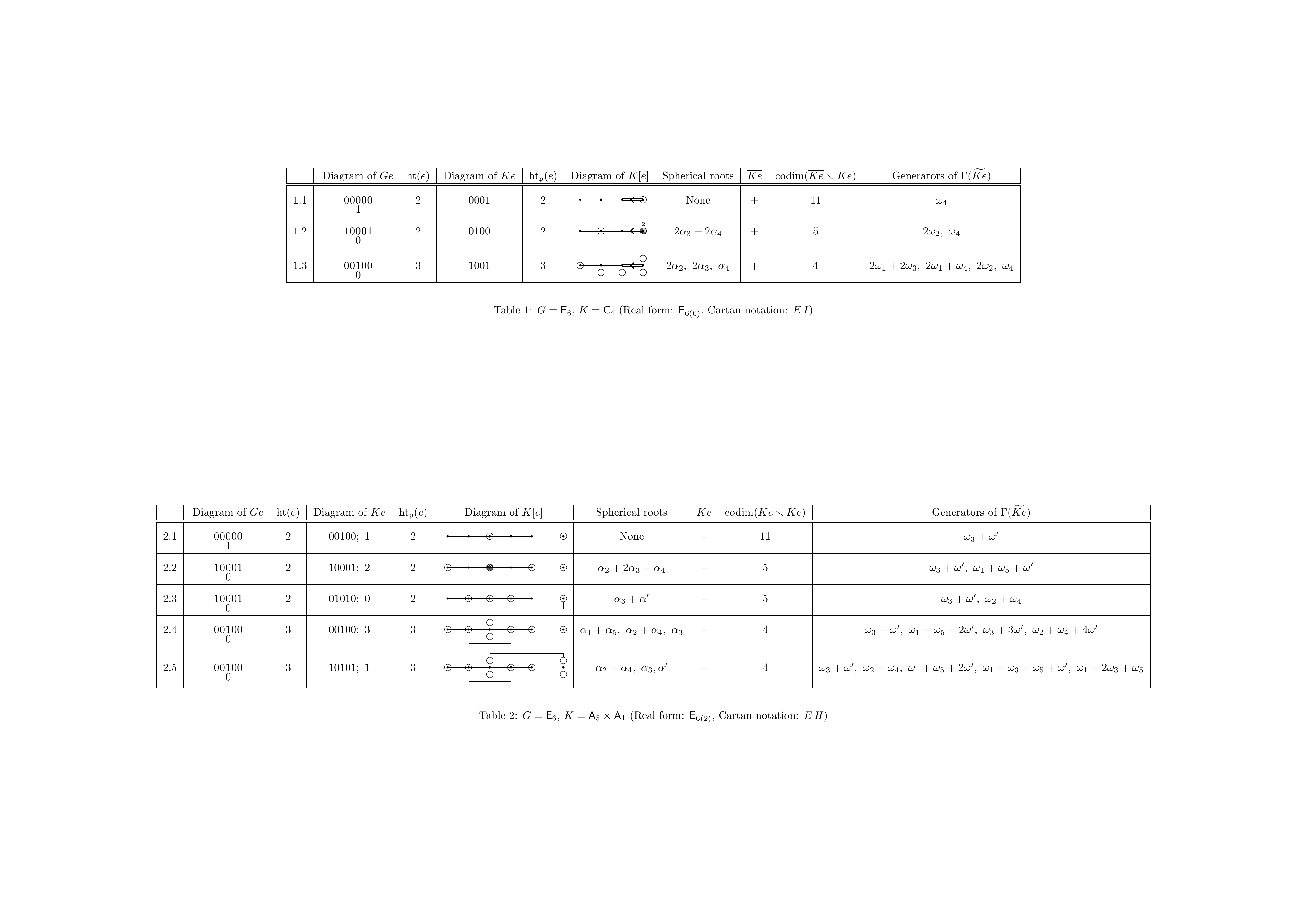}

\end{document}